\documentclass[a4paper, 11pt]{amsart}
\usepackage{enumerate}
\usepackage{amssymb}
\usepackage{tikz}
\usetikzlibrary{decorations.pathreplacing, decorations.markings, shapes}
\usepackage{graphicx}
\theoremstyle{plain}
\newtheorem{theorem}{Theorem}[section]
\newtheorem{lemma}[theorem]{Lemma}
\newtheorem{proposition}[theorem]{Proposition}
\newtheorem{corollary}[theorem]{Corollary}

\theoremstyle{definition}
\newtheorem{definition}[theorem]{Definition}
\newtheorem{notation}[theorem]{Notation}
\newtheorem{remark}[theorem]{Remark}
\makeatletter
\def\@map#1#2[#3]{\mbox{$#1 \colon #2 \longrightarrow #3$}}
\def\map#1#2{\@ifnextchar [{\@map{#1}{#2}}{\@map{#1}{#2}[#2]}}
\makeatother
\newcommand{\present}[1]{\ensuremath{\left\langle #1\right\rangle}}
\newcommand{\set}[2]{\ensuremath{\left\{#1 \,\colon #2\right\}}}
\newcommand{\labelarrow}[1]{\xrightarrow[\hspace*{10pt}]{#1}}
\newcommand{\evalat}[1]{\bigr\rvert_{#1}}
\providecommand{\abs}[1]{\lvert#1\rvert}
\providecommand{\norm}[1]{\lVert#1\rVert}
\DeclareMathOperator*{\Cay}{Cay}
\DeclareMathOperator*{\topp}{topp}
\newcommand{\cC}{\mathcal{C}}
\newcommand{\cS}{\mathcal{S}}
\newcommand{\SI}{\ensuremath{\mathbb{S}^1}}
\newcommand{\N}{\ensuremath{\mathbb{N}}}
\newcommand{\hvol}{h_{\text{vol}}}
\newcommand{\htop}{h_{\text{top}}}
\newcommand{\opp}[1]{\ensuremath{#1^{\text{opp}}}}
\newcommand{\sbmod}[2]{\ensuremath{\left[#1\right]_{#2}}}
\title[Volume entropy for surface groups]{Volume entropy for
minimal presentations of surface groups in all ranks}
\author{Llu\'{\i}s Alsed\`a, David Juher, J\'er\^ome Los \and Francesc
Ma\~{n}osas}
\address{Departament de Matem\`atiques, Edifici Cc, Universitat
Aut\`onoma de Bar\-ce\-lo\-na, 08913 Cerdanyola del Vall\`es,
Barcelona, Spain}
\email{alseda@mat.uab.cat}

\address{Departament d'Inform\`atica i Matem\`atica Aplicada,
Universitat de Girona, Llu\'{\i}s Santal\'o s/n, 17071 Girona, Spain}
\email{juher@ima.udg.edu}

\address{Aix-Marseille Universit\'e,
Institut Mathematiques de Marseille UMR 7373,
39 Rue F. Joliot Curie, 13013  Marseille,
France}
\email{los@cmi.univ-mrs.fr}

\address{Departament de Matem\`atiques, Edifici Cc, Universitat
Aut\`onoma de Bar\-ce\-lo\-na, 08913 Cerdanyola del Vall\`es,
Barcelona, Spain}
\email{manyosas@mat.uab.cat}

\thanks{The authors have been partially supported by MINECO grant
numbers MTM2008-01486 and MTM2011-26995-C02-01.
This work has been carried out thanks to the support of the
ARCHIMEDE Labex (ANR-11-LABX- 0033).}


\date{June 12, 2014}

\begin{document}
\begin{abstract}
We study the volume entropy of a class of presentations (including the
classical ones) for all surface groups, called \emph{minimal geometric
presentations}.
We rediscover a formula first obtained by Cannon and Wagreich
\cite{CW} with the computation in a non published manuscript by
Cannon \cite{Can80}.
The result is surprising: an explicit polynomial of degree $n$, the
rank of the group, encodes the volume entropy of all  classical
presentations of surface groups.
The approach we use is completely different.
It is based on a dynamical system construction following an idea due
to Bowen and Series \cite{BS} and extended to all geometric
presentations in \cite{Los}. The result is an explicit formula for the
volume entropy of minimal presentations for all surface groups,
showing a polynomial dependence in the rank $n>2$. We prove that for a
surface group $G_n$ of rank $n$ with a classical presentation $P_n$
the volume entropy is $\log(\lambda_n)$, where $\lambda_n$ is the
unique real root larger than one of the polynomial
\[
x^{n} - 2(n - 1) \sum_{j=1}^{n-1} x^{j} + 1.
\]
\end{abstract}
\maketitle

\section{Introduction}\label{sec:Intro}
In the beginning of the 80's several breakthroughs occurred in group
theory. The main one was the development of \emph{large scale
geometry} for groups, largely due to M. Gromov with, for instance,
the classification of groups with polynomial growth function
\cite{Gr2} or the introduction of the now standard notion of
\emph{hyperbolic groups} \cite{Gr1}. At about the same time R.
Grigorchuck \cite{Gri} found a class of groups with \emph{intermediate
growth function}. In all these classes of groups the \emph{growth
function} plays a central role. The growth function depends on the
generating set $X$ or on the presentation $P=\present{X/R}$ of the
group $G$. It is defined as the map $\N \mapsto \N$ such that
\[
n \mapsto f_{G,P}(n) = \operatorname{Card}
  \set{g \in G}{\operatorname{length}_X(g) \le n}.
\]
From the growth function $f_{G,P}$ several asymptotic functions are
defined such as the \emph{volume entropy} or the \emph{growth series}
also called the \emph{Poincar\'e series}.

The computational issues appeared also at about the same period. An
idea due to J. Cannon \cite{Can84} allows an inductive way to
describe geodesics in the Cayley graph $\Cay^1(G,P)$ via the notion
of \emph{cone types}. This notion has been intensively used later on
by Epstein, Cannon, Levy, Holt, Patterson, Thurston \cite{ECLHPT}
with the introduction of a very large class of groups, called
\emph{automatic}, that contains the hyperbolic groups of Gromov. The
computation of the growth function or the growth series becomes
possible in principle from a geodesic automatic structure, when it
exists. This is the case for hyperbolic groups. This computation, as
it was  noticed in \cite{CW}, can also be obtained using the
Floyd-Plotnick method \cite{FP}.

In practice, finding an explicit geodesic automatic structure from
the presentation is not so simple. For free groups with the free
presentation all the computations are easy and, for instance, the
volume entropy is simply $\log(2n-1)$, for the free group of rank
$n$ (see for instance \cite{DlH}). The next simple case is the class
of surface groups. For the classical presentations of surface
groups, the growth series appeared in a paper by Cannon and Wagreich
\cite{CW} without the explicit computation, leading to those series
that were earlier obtained in a non published manuscript of Cannon
\cite{Can80}. For hyperbolic groups, the existence of a geodesic
automatic structure for each presentation implies that the growth
series is a rational function (see \cite{ECLHPT, Can84}). In this
case the volume entropy (sometimes called the critical exponent) is
related to the largest pole of the growth series, i.e. the largest
root of the denominator of the growth series (see for instance
\cite{Cal}). The result of Cannon and Wagreich  for the classical
presentations of surface groups shows that the denominator $Q_{n}$
of the growth series is an explicit polynomial depending on the rank
$n\ge 3$ of the surface group:
\begin{equation}\label{thepolynomial}
Q_n(x) := x^{n} - 2(n - 1) \sum_{j=1}^{n-1} x^{j} + 1.
\end{equation}

The fact that a single, explicit polynomial could encode the volume
entropy for all surface groups is mysterious a priori, specially since
the original computations of Cannon did not appear in published form.

In this paper we rediscover the polynomial $Q_{n}(x)$ from a
completely different point of view and we hope that a part of the
mystery will disappear. In our approach we compute the volume entropy
of the group presentations from a dynamical system argument based on
an idea due to R. Bowen and C. Series \cite{BS} and generalized in
\cite{Los}.

The original idea of Bowen and Series was to associate a specific map
$\map{\Phi_{B-S}}{\SI}$ to a particular action of the group
$G = \pi_1(S)$ on the hyperbolic plane $\mathbb{H}^2$, where $\SI$ is
seen as the space at infinity of $\mathbb{H}^2$.
In \cite{Los}, $\SI$ is considered as the Gromov boundary of the
group $\partial G$ and a map $\map{\Phi_{P}}{\SI}$ is constructed for
each presentation $P$ in a class, called \emph{geometric},
characterized by the fact that the two dimensional Cayley complex
$\Cay^2(G, P)$ is planar. The maps $\Phi_{P}$ are called
\emph{Bowen-Series-Like} and they satisfy several interesting
properties, in particular the volume entropy of the presentation $P$
equals the topological entropy of the map $\Phi_{P}$. In addition the
map $\Phi_{P}$ admits a finite Markov partition and the computation of
the topological entropy for such maps is standard.

For any surface $S$, the classical presentation of the corresponding
surface group $\Gamma = \pi_1(S)$ is geometric. These classical
presentations are given by the minimal number of generators $n$ and
one relation of length $2n$. For orientable surfaces, $n$ is even and
equals $2g$, where $g$ is the genus of the surface. In this case, the
classical relation is a product of $g$ commutators. In the
non-orientable case, there is no restriction on the parity of $n$ and
the relation is given by the product of the squares of all generators
(see for instance~\cite{Sti}). A presentation with the minimal number
of generators is called \emph{minimal}. The rank 2 cases (torus and
Klein bottle) are, as usual, special: they are not hyperbolic, the
growth function is quadratic and thus the volume entropy is $0$. For
$n > 2$ all minimal geometric presentations are proved to have the
minimal volume entropy, among geometric presentations. It is
conjectured that this minimum should be an absolute minimum in
\cite{Los}.


We rediscover here the surprising explicit polynomial $Q_n(x),$
$r\ge 3$.
We will see that $Q_n(x)$ has a unique real root larger than one
denoted $\lambda_n$. More precisely we prove:

\begin{theorem}\label{MainTh}
For $n > 2$, let $\Gamma$ be a surface group of rank $n$ with a
minimal geometric presentation $P$. Then, the volume entropy of
$\Gamma$ with respect to the presentation $P$ is $\log(\lambda_n)$.
Moreover, for $n\ge 4$, $\lambda_n$ satisfies:
\[
2n-1 - \frac{1}{(2n-1)^{n-2}} < \lambda_n < 2n-1.
\]
\end{theorem}

The above inequalities show that the difference between the volume
entropy for the surface group and for the free group of the same rank
is explicitly very small.

The genesis of this rediscovery is interesting. The dynamical system
approach discussed above allows to compute the volume entropy of any
geometric presentation $P$ from an explicit Bowen-Series-Like map:
$\map{\Phi_{P}}{\SI}$. We developed an algorithm to compute the
entropy of such maps, for the classical presentations of orientable
surfaces, via the well known kneading invariant technique of Milnor
and Thurston \cite{MT}. The polynomial $Q_n(x)$ appears that way in
the computation for all orientable surfaces of genus $g\leq 43$. To
obtain the theorem we needed to compute the determinant of a matrix
whose size grows either linearly in $n$ with polynomial entries using
the Milnor-Thurston method, or quadratically in $n$ with integer
entries using the Markov matrix method. The computation leading to
the proof of the theorem became possible by a succession of two
surprises. First, by a particular choice of a minimal presentation
$P_n^{^+}$ the corresponding BSL map $\Phi_{P_n^{^+}}$ admits an
explicit  symmetry of order $2n$. By a quotient process, the Markov
matrix is reduced to an integer matrix whose size grows linearly in
$n$.
Then a method, developed in \cite{BGMY} under the nice name of the
``Rome technique'', was directly applicable to our case and reduced
the computation to a $2\times 2$ matrix with polynomial entries and no
computer was necessary.

The paper is organized as follows. In Section~\ref{sec:BSLMaps} we
recall the necessary ingredients for the construction of the map
$\Phi_{P}$ in some particular geometric presentations. The map is
then given explicitly for the particular minimal geometric
presentations with a symmetry property, together with its Markov
partition. In Section~\ref{sec:topent} we obtain a first formula for
the volume entropy in the orientable case, in terms of the Markov
matrix of the Bowen-Series-Like map. We exploit the symmetry of the
presentation to obtain a formula for the volume entropy in terms of
the spectral radius of a simpler matrix called the \emph{compacted
matrix of rank $n$}.
In Section~\ref{sec:disoriented} we extend these results to the
non-orientable case by showing that the volume entropy in this case
is also the logarithm of the spectral radius of the compacted
matrix of rank $n$.
The computation of the spectral radius of this new matrix is still
somewhat difficult. In Section~\ref{sec:CtoSC} we obtain a
simpler matrix with the same spectral radius.
Finally, in Section~\ref{sec:Rome}, the ``Rome method'' is
explained and applied to compute this spectral radius, and proving
Theorem~\ref{MainTh}.

\section{Bowen-Series-Like maps for geometric
presentations}\label{sec:BSLMaps}

In this section we review the necessary ingredients for the
construction of the Bowen-Series-Like maps defined in \cite{Los}.


\subsection{Geometric presentations}\strut\newline
Let $P = \present{X/R} = \present{x_1^{\pm 1}, x_2^{\pm 1},\dots,
x_n^{\pm 1}/R_1, \dots, R_k} $ be a presentation of a group
$\Gamma$. Recall that the Cayley graph $\Cay^1(\Gamma,P)$ is a
metric space and let $B_m$ be the ball of radius $m$ centred at the
identity. We denote the cardinality of any finite set $A$ by $|A|$.
The \emph{volume entropy} of $\Gamma$ with respect to the presentation $P$
is denoted by $\hvol(\Gamma,P)$ and defined as:
\[
  \lim_{m\to\infty} \frac{1}{m} \log \abs{B_m}.
\]

A presentation $P$ of a surface group $\Gamma=\pi_1(S) $ is called
\emph{geometric} if the Cayley 2-complex $\Cay^2(\Gamma,P)$ is a
plane. In particular the Cayley graph $\Cay^1(\Gamma,P)$ is a planar
graph. A geometric presentation $P$ is called \emph{minimal} if the
number of generators is minimal. For a group of an orientable
surface of genus $g$ it is well known that the minimal number of
generators is $2g$ (see \cite{Sti} for instance) and, in this case,
there is a presentation with a single relation of length $4g$. The
standard classical presentation in this case is the following:
\[
\present{x_1^{\pm 1}, y_1^{\pm 1}, x_2^{\pm 1},
                 y_2^{\pm 1}, \dots, x_g^{\pm 1}, y_g^{\pm 1}
                 / \prod _{i=1}^{g} [x_i, y_i]},
\]
where $[x_i, y_i] =  x_i \cdot y_i \cdot x_i^{-1} \cdot y_i ^{-1}$
is a commutator.

For a rank $n$ group of a non-orientable surface there is also a
classical presentation with a single relation of length $2n$:
\[
\present{x_1^{\pm 1}, x_2^{\pm 1}, \dots, x_n^{\pm 1}
                 / \prod _{i=1}^{n} x_i^2}.
\]

It is easy to check that such classical presentations are geometric
(see below).


Geometric presentations satisfy very simple combinatorial
properties:

\begin{lemma}[Floyd and Plotnick \cite{FP}]\label{old-lem:2.2}
If $P = \present{x_1^{\pm 1}, \dots, x_n^{\pm 1} / R_1, \dots, R_k}$
is a geometric presentation of a surface group $\Gamma$ then $P$
satisfies the following properties:
\begin{enumerate}
\item The set $\{x_1^{\pm 1}, \dots, x_n^{\pm 1}\}$ admits a
cyclic ordering that is preserved by the $\Gamma$-action.

\item Each generator appears exactly twice (with plus or minus
exponent) in the set $R = \{R_1, \dots, R_k\}$ of relations.

\item Each pair of adjacent generators, according to the cyclic
ordering~(a), appears exactly once in $R$ and defines uniquely a
relation $R_i \in R$.
\end{enumerate}
\end{lemma}

The following statement is the main ingredient to compute the volume
entropy of a geometric presentation. The statement also contains the
main result about minimal geometric presentations. In what follows
$\SI$ will denote a (topological) circle. Recall that any surface
group $\Gamma$ is Gromov-hyperbolic \cite{Gr1} and its boundary is:
$\partial \Gamma\simeq\SI$.

Let us introduce the notion of a Markov partition. Let $W$ be a
finite set of $\SI.$ An interval of $\SI$ will be called
\emph{$W$-basic} if it is the closure of a connected component of
$\SI\setminus W$. Observe that two different $W$-basic intervals
have pairwise disjoint interiors. Let {\map{\phi}{\SI}} and let $W
\subset \SI$ be finite. We say that $W$ is a \emph{Markov partition
of $\phi$} if $W$ is $\phi-$invariant (i.e., $\phi(W) \subset W$)
and the image by $\phi$ of every basic interval is a union of basic
intervals.

\begin{theorem}[Los \cite{Los}]\label{old-theo:2.3}
Let $P$ be a geometric presentation of a surface group $\Gamma$.
Then there exists a map {\map{\Phi_{P}}{\partial\Gamma = \SI}} with
the following properties:
\begin{enumerate}
\item The map $\Phi_{P}$ is Markov, i.e. it admits a finite Markov
partition.

\item The topological entropy of $\Phi_{P}$, $\htop(\Phi_{P}),$ is
equal to the volume entropy $\hvol(\Gamma,P)$.
\end{enumerate}
In addition, the volume entropy is minimal, among geometric
presentations, for all minimal geometric presentations.
\end{theorem}

The map $\Phi_{P}$ satisfies more properties that are not needed
here. Property~(a) is specially interesting for computations since,
for Markov maps, it is classical that the topological entropy is
nothing but the logarithm of the spectral radius of a finite integer
matrix, the Markov transition matrix (see \cite{Shu} or \cite{ALM}
for instance). The goal of the next sections is to make such a
Markov partition explicit in the particular cases of minimal geometric
presentations.

\subsection{Construction of the Bowen-Series-Like map}\strut

In this subsection we review the definition and the necessary
properties of the BSL maps, in the particular case of minimal
geometric presentations.

\subsubsection{Bigons}
As we have seen, a presentation $P = \present{X/R}$ defines the
Cayley graph $\Cay^1(\Gamma,P)$ and the Cayley 2-complex
$\Cay^2(\Gamma,P)$. A \emph{bigon} in $\Cay^1(\Gamma,P)$ is a pair
of distinct geodesics $\{\gamma_1, \gamma_2\}$ connecting two
vertices $\{v,v'\} \in \Cay^1(\Gamma,P)$. We denote by $B_v(x,y)$
the set of bigons $\{\gamma_1, \gamma_2\}$ whose initial vertex is
$v$ and so that the geodesic $\gamma_1$ starts at $v$ by the edge
labelled $x$ and $\gamma_2$ starts at $v$ by the edge labelled $y$,
with $x\neq y$. By the $\Gamma$-action we can fix the initial vertex
$v$ to be the identity and we denote $B_{\operatorname{id}}(x,y)$ by
$B(x,y)$.

For geometric presentations of surface groups the set of bigons is
particularly simple.

\begin{lemma}\label{old-lem:2.4}
If $P = \present{X/R}$ is a geometric presentation of a surface
group $\Gamma$ then the set of bigons $B(x,y)$ is non empty if and
only if $(x,y)$ is an adjacent pair of generators, according to the
cyclic ordering of Lemma~\ref{old-lem:2.2}(a). In addition, if
$(x,y)$ is an adjacent pair of generators there is a unique bigon
$\beta(x, y) \in B(x,y)$ of finite minimal length, called
\emph{minimal bigon}.
\end{lemma}

This lemma is proved in \cite[Lemmas~2.6 and 2.12]{Los}. Observe
that each minimal bigon is particularly simple for a geometric
presentation where all relations have even length. Indeed, each pair
of adjacent generators $(x,y)$ defines a unique relation by
Lemma~\ref{old-lem:2.2}(c). The relation can be written, up to
cyclic permutation and inversion ($R_i \rightarrow (R_i )^{-1}$),
as: $R_i = \gamma_1 \cdot (\gamma_2 )^{-1}$, where $\gamma_1$ is a
word (or a path) starting by the letter $x$, $\gamma_2$ starts by
the letter $y$ and $l(\gamma_1) = l(\gamma_2)$. The observation is
that for geometric presentations the two paths $\gamma_1$  and
$\gamma_2$ start at the identity and end at the same vertex (since
$R_i$ is a relation) and are geodesics. In other words the pair
$\{\gamma_1, \gamma_2\}$ is a bigon. It is minimal and unique by
Lemma~\ref{old-lem:2.2}(c).

\subsubsection{Bigon-Rays}\label{BigonRays}
We describe a canonical way to define a point on the boundary
$\partial\Gamma$ associated to an adjacent pair of generators
$(x,y)$. Recall that a surface group is hyperbolic in the sense of
Gromov \cite{Gr1} and its boundary $\partial\Gamma$ is the circle
$\SI$. By definition of $\partial\Gamma$, a point $\xi \in
\partial\Gamma$ is the limit of geodesic rays, for instance starting
at the identity, modulo the equivalence relation among rays that two
rays are equivalent if they stay at a uniform bounded distance from
each others (c.f. \cite{Gr1}). If $\xi \in \partial\Gamma$ is a
point on the boundary we denote by $\{\xi\}$ a geodesic ray starting
at identity and converging to $\xi$.

In what follows, given two integers $k$ and $l$ we will denote $k
\pmod{l}$ by $\sbmod{k}{l}$. Also, we choose $1,2,\dots, l$ as the
representatives of the classes modulo $l;$ that is, $\sbmod{0}{l} =
\sbmod{l}{l} = l.$ However, unless necessary we omit the modulo part
in the notations.

\begin{notation}\label{notationCicOrd}
In what follows we denote the $n$ generators (and their inverses) by
$y_1, y_2, \dots, y_{2n}$ in such a way that $y_{\sbmod{i \pm
1}{2n}}$ are the elements adjacent to $y_i$ with respect to the
cyclic ordering from Lemma~\ref{old-lem:2.2}(a). We denote an
adjacent pair by $(y_i, y_{\sbmod{i+1}{2n}})$ where, by convention,
the edges denoted $y_i$ and $y_{\sbmod{i+1}{2n}}$ are adjacent and
oriented from the vertex. We also adopt the convention that $y_i$ is
on the \emph{left} of $y_{\sbmod{i+1}{2n}}$ (see
Figure~\ref{fig:notationCicOrd}). This convention defines an
orientation of the plane $\Cay^2(\Gamma,P)$.

The parity of the number of adjacent pairs at each vertex implies
that $(y_i, y_{\sbmod{i+1}{2n}})$ defines an \emph{opposite} pair,
with respect to the cyclic ordering of Lemma~\ref{old-lem:2.2}(a),
defined by:
\[
\opp{(y_i, y_{[i+1]_{2n}})} :=
  (y_{\sbmod{i+n}{2n}}, y_{\sbmod{i+n+1}{2n}})
\]
(see Figure~\ref{old-fig:1}).
\end{notation}
\begin{figure}
\begin{tikzpicture}
\tikzstyle{post}=[->,>=stealth,semithick]
\tikzset{->-/.style={decoration={markings,mark=at position #1 with {\arrow{>}}},postaction={decorate}}}
\foreach \angle / \lab in { 90/1, 45/2, 315/n }{
   \draw[post,->-=0.6] (0,0) to node[pos=0.8, below, rotate=\angle] {$y_{_{\lab}}$} (\angle:2);
}
\foreach \angle in {325,335,...,385} { \node at (\angle:1.4) {$\cdot$}; \node at (\angle+180:1.5) {$\cdot$}; }
\foreach \angle / \lab in { 270/n+1, 225/n+2, 135/2n }{
   \draw[post,->-=0.6] (0,0) to node[pos=0.8, above, rotate=(\angle-180)] {$y_{_{\lab}}$} (\angle:2);
}
\end{tikzpicture}
\caption{The labelling of the generators (and the cyclic ordering) fixed in
Notation~\ref{notationCicOrd}.}\label{fig:notationCicOrd}
\end{figure}
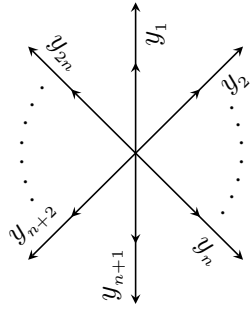

We construct a unique infinite sequence of adjacent pairs, bigons
and vertices from any given pair $(y_i, y_{\sbmod{i+1}{2n}})$ by the
following process:
\begin{enumerate}[{\textbf{Step\ }1{.}}]
\item Each adjacent pair, at the identity, defines a unique minimal
bigon $\beta(y_i, y_{i+1})$ by Lemma~\ref{old-lem:2.4}.
The bigon $\beta(y_i, y_{i+1})$ is a pair of ge\-o\-des\-ics
$\{\gamma_{l}, \gamma_{r}\}$, where the indices $l, r$  stand for left
and right, with respect to an orientation of the plane $\Cay^2(\Gamma, P)$.
The geodesics $\{\gamma_{l}, \gamma_{r}\}$ connect the identity to a
vertex $v_1 = v_1[\beta(y_i, y_{i+1})].$

\item The two geodesics $\{\gamma_{l}, \gamma_{r}\}$ end at $v_1$ by
two generators that are adjacent by Lemma~\ref{old-lem:2.4}. Therefore
the bigon $\beta(y_i, y_{i+1})$  defines a unique adjacent pair at
$v_1$, called a \emph{top pair} of $\beta(y_i, y_{i+1})$, which is
denoted:
$\topp[\beta(y_i, y_{i+1})]$, based at $v_1 = v_1[\beta(y_i,y_{i+1})]$
and is uniquely defined by $(y_i, y_{i+1}).$

\item The pair $\topp[\beta(y_i, y_{i+1})]$ defines an opposite pair
at $v_1$, denoted by:
\[
\opp{\bigl(\topp[\beta(y_i, y_{i+1})]\bigr))} := (y_i,
y_{i+1})^{(1)}.
\]

\item We consider then the unique minimal bigon, at $v_1$, defined by
the pair $(y_i, y_{i+1 })^{(1)}$ by Lemma~\ref{old-lem:2.4}:
\[
\beta^{(1)}(y_i, y_{i+1}) := \beta_{v_1}[(y_i, y_{i+1 })^{(1)}].
\]

\item The bigon  $\beta^{(1)}(y_i, y_{i+1})$ defines a new top pair
$\topp[\beta^{(1)}(y_i, y_{i+1})]$, at the vertex $v_2$.
\end{enumerate}

The Steps 1--5 define, by induction, a unique infinite sequence
of vertices and bigons (see Figure~\ref{old-fig:1}):
\begin{equation}\label{old-eq:7}
\begin{split}
 & \operatorname{id}, v_1, v_2, \cdots\\
 & \beta(y_i, y_{i+1}), \beta^{(1)}(y_i,y_{i+1}),
      \beta^{(2)}(y_i, y_{i+1}) \cdots\ .
\end{split}
\end{equation}
Each bigon in the infinite sequence
$\set{\beta^{(k)}(y_i, y_{i+1})}{k \in \N}$
is a pair of geodesics
$\left\{\gamma^{(k)}_{l}, \gamma^{(k)}_{r}\right\}$ with $k \in \N$
connecting the vertices $v_{k}$ and $v_{k+1}$.

By definition, the terminal vertex $v_{k+1}$ of $\beta^{(k)}$ is the
initial vertex of the next bigon $\beta^{(k+1)}$ in the sequence.
Therefore a \emph{finite concatenation} of bigons $ \beta^{(0)}(y_i,
y_{i+1}) \beta^{(1)}(y_i, y_{i+1}) \cdots \beta^{(k)}(y_i, y_{i+1})$
makes sense. It is defined by the finite collection of paths:
\[
\set{  \gamma^{(0)}_{\epsilon(0)} \cdot
       \gamma^{(1)}_{\epsilon(1)} \cdots
       \gamma^{(k)}_{\epsilon(k)}}{
  \text{$\epsilon(j)\in\{l,r\},\ j \in \{0, 1, \dots, k\}$}} .
\]

We denote the infinite concatenation of all these paths as:
\[
\beta^{\infty}(y_i, y_{i+1}) := \lim_{k\to\infty}
\beta^{(0)}(y_i, y_{i+1}) \beta^{(1)}(y_i, y_{i+1}) \cdots
\beta^{(k)}(y_i, y_{i+1}).
\]

\begin{figure}[htbp]
\centerline{\includegraphics[height=65mm]{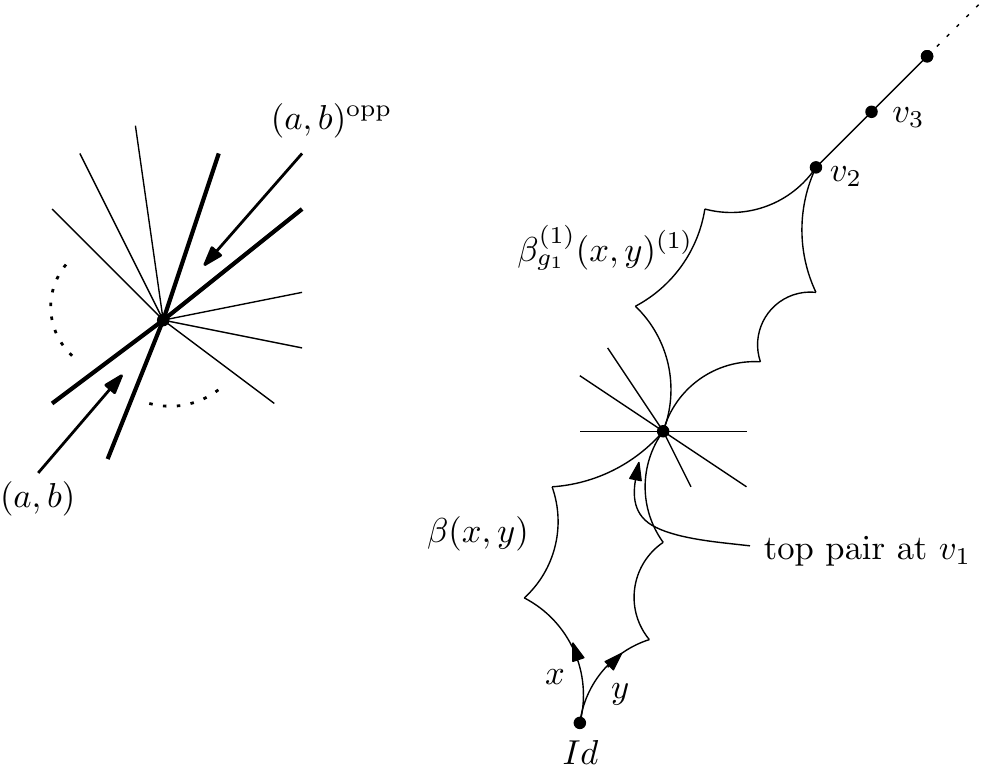} }
\caption{Opposite pair and bigon rays.}\label{old-fig:1}
\end{figure}

\begin{lemma}[Los \protect{\cite[Lemma~3.1]{Los}}]\label{old-lem:2.5}
With the above notation the following statements hold.
\begin{enumerate}
\item Each path in the collection:
$
\beta^{(0)}(y_i, y_{i+1}) \beta^{(1)}(y_i, y_{i+1}) \cdots
\beta^{(k)}(y_i, y_{i+1})
$
is a geodesic segment, for all $k\in \N$.

\item Two geodesic segments in (a) stay at a uniform distance from
each other for any $k \in \N$.
\end{enumerate}
\end{lemma}

In consequence, the infinite concatenation $\beta^{\infty}(y_i,
y_{i+1})$ defines infinitely many geodesic rays with a unique limit
point in $\partial\Gamma.$ It will be denoted by $(y_i,
y_{i+1})^{\infty}$.

\subsubsection{Cylinders, definition of the BSL map}
We define the \emph{cylinder} of length one as the subset of the
boundary:
\[
\cC_x := \set{\xi \in \partial\Gamma}{
  \text{there is a geodesic ray $\{\xi\}$ starting at
        $\operatorname{id}$ by $x \in X$}}.
\]

\begin{lemma}
Let $P=\present{X/R}$ be a geometric presentation of $\Gamma$. The
boundary $\partial\Gamma = \SI$ is covered by the cylinder sets
$\cC_x,\ x\in X$ and:
\begin{enumerate}
\item Two cylinders have non-empty intersection:
$\cC_x \bigcap \cC_y \neq \emptyset$ if and only if
$(x, y)$ is an adjacent pair of generators.

\item Each cylinder $\cC_x,\ x\in X$ is a non trivial
connected interval of $\partial\Gamma$.
\end{enumerate}
\end{lemma}

This lemma is proved in \cite[Lemmas~2.13 and 2.14]{Los}.
Observe that the point $(y_i, y_{i+1})^{\infty}$ of
Lemma~\ref{old-lem:2.5}
belongs, by definition, to the intersection
$\cC_{y_i} \bigcap \cC_{y_{i+1}}$.

In what follows we consider the points in the circle ordered
\emph{clockwise}.\label{clockwise}
That is, if $r,s,t$ are pairwise different points of $\SI$ we will
write $r < s < t$ if $s$ belongs to the clockwise arc starting at
$r$ and ending at $t$.
The notation $r \le s \le t$ will also be used in the natural way.
Then the interval $[r,t]$ is defined as the set
$\set{s \in \SI}{r \le s \le t}.$
Also, if $I, J, K$ are closed connected subsets of $\SI$
with pairwise disjoint interiors we will write $I < J < K$ whenever
$r \le s \le t$ for every $r\in I,\ s\in J$ and $t \in K$.

\begin{definition}\label{BSLM}
If $P = \present{X/R}$ is a geometric presentation of a hyperbolic
surface group $\Gamma$, then we denote by $I_{y_i}$ the interval
$[(y_{i-1}, y_{i})^{\infty}, (y_{i}, y_{i+1} )^{\infty}].$ Clearly
$I_{y_i}$ is a subset of $ \cC_{y_i}$ for every $y_i \in X$.

We define the Bowen-Series-Like map {\map{\Phi_P}{\partial\Gamma}} by
\[
\Phi_P(\xi ) = x^{-1}(\xi) \quad \text{if $\xi \in I_x$,}
\]
where $ x^{-1} (\xi)$ is the action, by homeomorphism, on
$\partial\Gamma$ by the group element $x^{-1}$.

The map $\Phi_P$ satisfies the following elementary properties:
\begin{enumerate}[(i)]
\item It depends explicitly on the presentation $P$ (the exact
dependence will be explained below).

\item Since $I_{x} \subset \cC_{x},$ each $\xi \in I_{x}$ has a
writing, as a limit of a ray, as
$\{\xi\} = x \cdot \omega.$
The image under $\Phi_P$ is given by:
\[
 \{\Phi_P(\xi)\} = \{ x^{-1}(x \cdot \omega) \} = \{\omega\}.
\]
In other words, the map $\Phi_P$ is a shift map, on this particular
writing as a ray.
\end{enumerate}
\end{definition}

\subsection{Markov partition for minimal geometric presentations}\strut

Theorem~\ref{old-theo:2.3} states that the map $\Phi_P$ admits a
Markov partition. In this subsection we will define a particular
presentation, which will be called \emph{symmetric}, and we will
make the Markov partition explicit for this presentation.

The first step is to define subdivision points in each interval
$I_x$, $x\in X$. Let us recall that the extreme points $(y,
x)^{\infty}$ and $(x, z)^{\infty}$ of the intervals $I_x$ are limit
points of bigon rays $\beta^{\infty}(y, x)$ and $\beta^{\infty}(x,
z)$. Let us focus on $(y, x)^{\infty}$. Let
$\beta_{v}^{\infty}(y, x)$ be the bigon ray starting at the vertex $v
\in \Cay^1(\Gamma,P)$. Observe that with this definition we can
write:

\begin{equation}\label{old-eq:8}
\beta^{\infty}(y, x) = \beta(y, x) \cdot
\beta_{v_1}^{\infty}[(y, x)^{(1)}],
\end{equation}
with the notations of Subsection~\ref{BigonRays}.

The particular property of a minimal geometric presentation that is
useful here is that there is only one relation $R$ of even length
$2n$, when $\Gamma$ is a surface group of rank $n$. In this case,
any bigon $\beta(y, x)$ has the form $\{\gamma_l, \gamma_r\}$ with
$\gamma_l \cdot (\gamma_r)^{-1}$ being one of the words representing
the relation $R$, up to cyclic permutation and inversion. This word
starts with the letter $y$ and terminates with the letter $x^{-1}$.

Since the relation $R$ has length $2n$, let us write the two paths
$\{\gamma_l, \gamma_r\}$ as:
\begin{equation}\label{old-eq:9}
\{ y \cdot x'_{i_2} \cdots x'_{i_n},
   x \cdot x_{i_2} \cdots x_{i_n}
\}.
\end{equation}
We focus on the ``$x$'' side of
Equations~\eqref{old-eq:8},\eqref{old-eq:9}, i.e. on the infinite
collection of rays:
\begin{equation}\label{old-eq:10}
x \cdot x_{i_2} \cdots x_{i_n} \cdot \beta^{(\infty)}_{v}[(y, x)^{(
1)}],
\end{equation}
where $v$ is the group element written: $v = x \cdot x_{i_2} \cdots
x_{i_n}$. The vertices $v^1 = x$ and $v^{j} =  x \cdot x_{i_2}
\cdots x_{i_j},$ for $j=2,3,\dots, n-1$ of $\Cay^1(\Gamma,P)$ belong
to $\gamma_r$ and are ordered along $\gamma_r$ (this notation is
consistent with $v = v_n$ ).

The following pairs of consecutive edges:
\begin{equation}\label{old-eq:12}
\left\{
   (\overline{x}, x_{i_2}),
   (\overline{x_{i_2}}, x_{i_3}), \dots,
   (\overline{x_{i_{n-1}}}, x_{i_n})
\right\}
\end{equation}
at the vertices $\left\{v^1, \dots, v^{n-1}\right\}$, are
\emph{crossed} by the path $\gamma_r$, where the notation
$\overline{x_{i_j}}$ means the edge  $x_{i_j}$ with the opposite
orientation. Observe that each pair of consecutive letters along the
paths $\gamma_r$ are adjacent generators.

\begin{lemma}\label{old-lem:2.8}
If the relation defining $\beta(y, x)$ has even
length $2n$ then the collection:
\begin{equation}\label{old-eq:13}
\mathcal{R}_L^{x} := \set{ x \cdot x_{i_2} \cdots x_{i_j} \cdot
\beta^{(\infty)}_{v^{j}}
[\opp{(\overline{x_{i_{j}}},x_{i_{j+1}})}]}{ j = 1, \dots, n-1},
\end{equation}
(see Figure~\ref{old-fig:2}) is called the \emph{left} (with respect
to $x$) \emph{subdivision rays}. They satisfy the following
properties:
\begin{enumerate}
\item Each path in the infinite collection $\mathcal{R}_L^{x}$ is a
ray starting at the identity.

\item For a given $j \in \{1, 2, \dots, n-1\}$, all the rays in
\[ \mathcal{R}_L^{(x,j)} =
  x \cdot x_{i_2} \cdots x_{i_j} \cdot
  \beta^{(\infty)}_{v^{j}} [\opp{(\overline{x_{i_{j}}}, x_{i_{j+1}})}]
\]
converge to the same point $\lambda_{x}^{j} \in \partial\Gamma$.

\item For any $j \neq p$, the rays in $\mathcal{R}_L^{(x,j)}$ and
in $\mathcal{R}_L^{(x,p)}$ have a common beginning:
$x \cdot x_{i_2} \cdots x_{i_{\nu}}$
where $\nu:= \min\{j, p\}$ and are otherwise disjoint.

\item Each $\lambda_{x}^{j},\ j \in \{1,2, \dots, n-1\}$ belongs to
the interior of the interval $I_{x}$ of Definition~\ref{BSLM}.

\item The limit points $\lambda_{x}^{j}$ are inversely ordered with
respect to the index $j \in \{1,2, \dots, n-1\}$ along
$\partial\Gamma$ (that is, $
 \lambda_x^{n -1} < \lambda_x^{n -2} < \dots <
 \lambda_x^{2} <  \lambda_x^{1}
$%
).
\end{enumerate}
\end{lemma}

This lemma is proved in \cite[Lemma~4.1]{Los}.

\begin{figure}[htbp]
\centerline{\includegraphics[height=65mm]{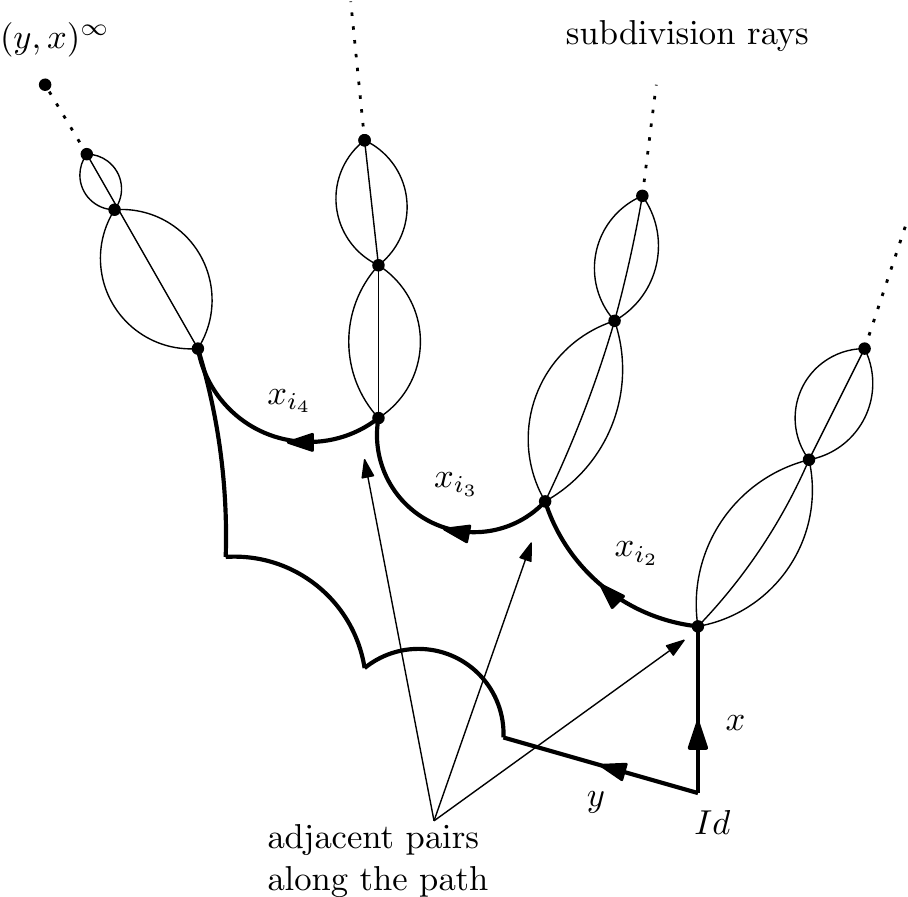}}
\caption{Subdivision rays.}\label{old-fig:2}
\end{figure}

We denote $\mathcal{L}_{x} = \{\lambda_{x}^{1}, \dots,
\lambda_{x}^{n-1}\}$ this set of \emph{left} (with respect to $x$)
limit points. By the same analysis the adjacent pair $(x, z)$
defines the set of \emph{right} (with respect to $x$) limit points
$\mathcal{R}_{x} = \{\rho_{x}^{1}, \dots, \rho_{x}^{n-1}\},$ which
are ordered with respect to the superindex. Observe that we use here
the fact that a minimal geometric presentation has only one relation
(of length $2n$). Consider now the set of all such points:
\begin{equation}\label{old-eq:14}
\cS = \bigcup_{x\in X} \left(
  \mathcal{R}_{x} \cup \mathcal{L}_{x} \cup \partial I_x
\right),
\end{equation}
called the \emph{subdivision points}.

\begin{lemma}\label{old-lem:2.9}
If $P$ is a geometric presentation of a hyperbolic surface group
$\Gamma$ so that all relations have even length, then the set of
subdivision points $\cS$ is invariant under the map $\Phi_P$ of
Definition~\ref{BSLM} and defines a finite Markov partition of
$\partial\Gamma$.
\end{lemma}

This statement is a particular case of \cite[Theorem~4.3]{Los}. For
a minimal geometric presentation there is only one relation of
length $2n$ for a surface group of rank $n$. In this case the
partition of each interval $I_x$ above is given by the points
$\mathcal{R}_{x} \cup \mathcal{L}_{x} \cup \partial I_x$ which are
ordered in the following way:
\[
  \lambda_x^{n} := (y, x)^{\infty} < \lambda_x^{n-1} < \dots <
   \lambda_x^{2} < \rho_x^{1} < \lambda_x^{1} < \rho_x^{2} < \dots <
   \rho_x^{n} := (x, z)^{\infty}\ .
\]
We also observe here that the intervals $I_x$ are ordered, along
$\SI$, by the (cyclic) ordering of the generators at the identity.
Then, we can define a partition of each of the intervals $I_x$
consisting on the following subintervals:
\begin{equation}\label{eq:PartInt}
\begin{split}
& L_x^j = \left[\lambda_x^{j}, \lambda_x^{j-1}\right]
  \text{ and } R_x^j = \left[\rho_x^{j-1}, \rho_x^{j}\right],
  \text{ for $j \in \{3, 4, \dots, n\}$,}\\
& C_x^L = \left[\lambda_x^{2}, \rho_x^{1}\right]
  \text{ and } C_x^R = \left[\lambda_x^{1}, \rho_x^{2}\right], \text{ and}\\
& C_x = \left[\rho_x^{1}, \lambda_x^{1}\right].
\end{split}
\end{equation}

Recall that a subdivision point (left or right) has the following
writing:
$
\{\lambda_x^j\} =
  x \cdot x_{i_2} \cdots x_{i_j} \cdot
  \beta_{v^j}^{\infty}[\opp{(\overline{x_{i_{j}}}, x_{i_{j+1}})}],
$ for $j \in \{1,2, \dots, n\}$.

Since the map $\Phi_P$ acts, on each interval $I_x$, on the ray
writing as a shift map we obtain:
\begin{equation}\label{old-eq:18}
\begin{split}
\{\Phi_P(\lambda_x^1)\} &=
  \beta^{\infty}[\opp{(\overline{x}, x_{i_2})}], \text{and}\\
\{\Phi_P(\lambda_x^j)\} &=
  x_{i_2} \cdots x_{i_j} \cdot \beta_{x_{i_2} \cdots
  x_{i_j}}^{\infty} [\opp{(\overline{x_{i_{j}}}, x_{i_{j+1}})}]
  \text{ for $j \in \{2, 3, \dots, n\}$}
\end{split}
\end{equation}
and there is a similar writing for the points $\rho_x^j$.

\begin{lemma}\label{old-lem:2.10}
If $P$ is a geometric presentation of a surface group with all
relations of even length then the image of the \emph{central}
interval $C_x = [\rho_x^{1}, \lambda_x^{1}]$ under $\Phi_P$ is a
single interval $I_u,\ u\in X$, where $u$ is the generator that is
opposite to $x^{-1}$ for the cyclic ordering of
Lemma~\ref{old-lem:2.2}(a) at the vertex $x$.
\end{lemma}

\begin{proof} By \eqref{old-eq:18} we observe that
$
\{\Phi_P(\lambda_x^1)\} =
  \beta^{\infty}[ \opp{(\overline{x}, x_{i_2})}]
$
and similarly
$
\{\Phi_P(\rho_x^1)\} =
  \beta^{\infty}[\opp{(x'_{i_2}, \overline{x})}].
$
Since the two adjacent pairs $(\overline{x}, x_{i_2})$ and $(x'_{i_2},
\overline{x})$ are adjacent at the vertex $ v_1 = x$ then the two opposite
pairs $\opp{(\overline{x}, x_{i_2})}$ and
$\opp{(x'_{i_2},\overline{x})}$ are also adjacent. That means that they share one edge
$u$. This edge is just the one that is opposite to $x^{-1}$ at the
vertex $x$ (see Figure~\ref{old-fig:CI}).
\end{proof}

\begin{figure}[htbp]
\centerline{\includegraphics[height=60mm]{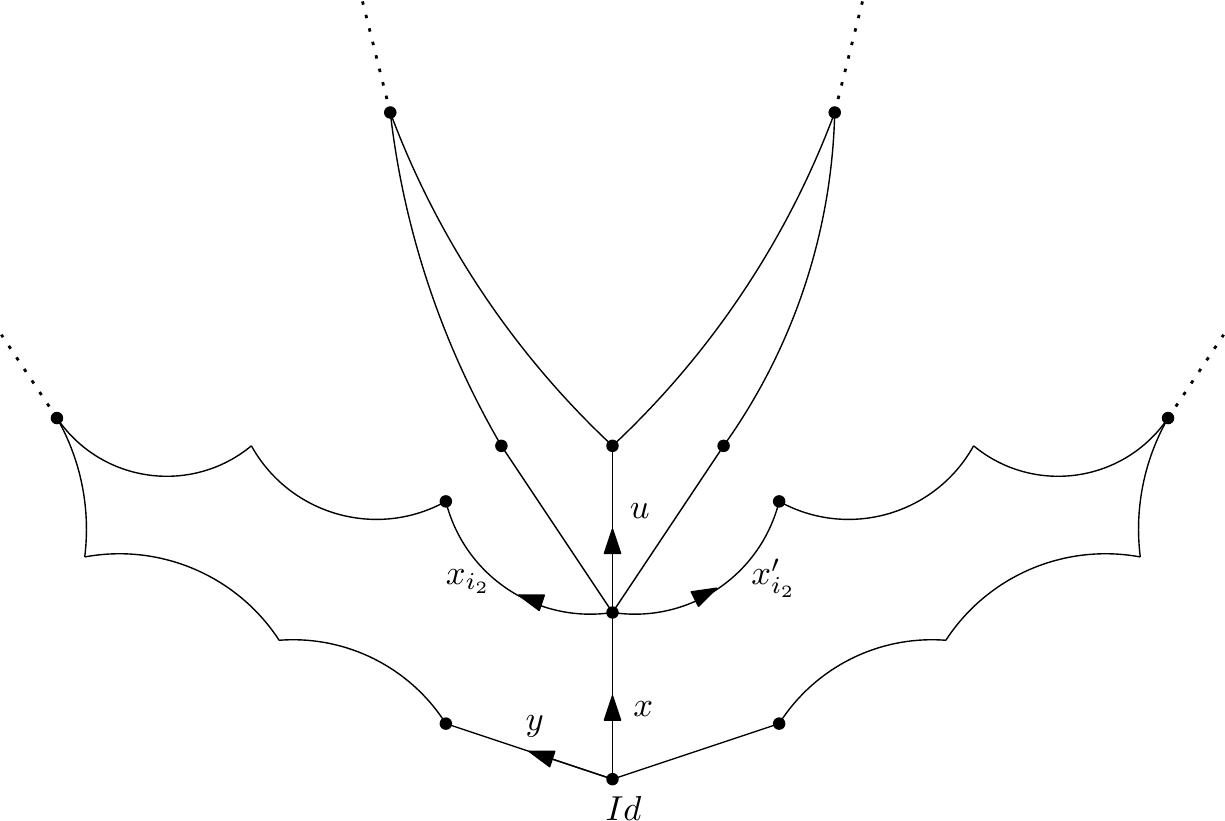} }
\caption{Central Interval.}\label{old-fig:CI}
\end{figure}

Next we define a particular presentation, which we call
\emph{symmetric}, for the rank $n$ group $\pi_1(S)$, where $S$ is an
orientable surface. Recall that $n=2g$, where $g$ is the genus of
$S$.

\begin{definition}\label{SimPres}
Given a surface group $\pi_1(S_g)$ of rank $n=2g$,
where $S_g$ is orientable of genus $g,$ the presentation
\[
\present{x_1^{\pm 1}, x_2^{\pm 1}, \dots, x_{n}^{\pm 1} \left/
        x_1 x_2 \cdots x_{n}
        x_1^{-1} x_2^{-1} \cdots x_{n}^{-1} \right.}
\]

will be called \emph{symmetric} and denoted by $P_n^{^+}.$
\end{definition}

\begin{proposition}\label{old-prop:2.1}
The symmetric presentation $P_n^{^+}$ is minimal and geometric.
\end{proposition}

\begin{proof} Consider the polygon $\Delta_n$ with $2n$ sides,
labelled by the elements of
$\{x_1^{\pm1},x_2^{\pm1},\ldots,x_{n}^{\pm1}\}$ in the ordering:
\[
x_1 \cdot x_2 \cdots  x_{n} \cdot x_1^{-1} \cdot x_2^{-1} \cdots
x_{n}^{-1}.
\]
The identification of the side labelled $x_i$ with the one labelled
$x_i^{-1}$ defines an orientable surface of genus $g$. The
identification is an equivalence relation $\sim$ and $\Delta_n/\sim$
is the surface of genus $g$. The presentation $P_n^{^+}$ is minimal
since it has $n$ generators and it is geometric because the universal
cover of the surface $\Delta_n/\sim$ is nothing but the Cayley
2-complex $\Cay^2(\Gamma,P_n^{^+})$ that is a plane.
\end{proof}

Lemma \ref{old-lem:2.2} says that for geometric presentations the
generators have a cyclic ordering at each vertex. For the
presentation $P_n^{^+}$ the cyclic ordering is
\[
  x_1 < x_2^{-1} < x_3 < x_4^{-1}< \cdots <
  x_{n-1}< x_{n}^{-1}<x_1^{-1} <x_2 < \cdots <
  x_{n-1}^{-1}< x_{n}.
\]

\section{The topological entropy of the map
$\Phi_{P_n^{^+}}$}\label{sec:topent}

The aim of this section is to start the computation of the
topological entropy of the Bowen-Series-Like map $\Phi_{P_n^{^+}}$ for
the symmetric presentation $P_n^{^+}=\present{X/R}$ of the orientable
surface group of rank $n$.

Since the surface is orientable and the presentation is
geometric and minimal, then all generators $x \in X$ act on
$\partial \Gamma$ as an orientation preserving homeomorphism.
By Definition~\ref{BSLM}(ii), $\Phi_{P}\evalat{I_{x}}$ is an
orientation preserving homeomorphism for every $x\in X$ and from
Lemma~\ref{old-lem:2.9} the set $\cS$ defines a Markov
partition of $\Phi_{P}$. Since
$
\partial I_{x_i} \subset \cS
$
we also have that $\Phi_{P}$ is a homeomorphism on every
$\cS-$basic interval.

In this situation the topological entropy can be easily computed
as the logarithm of the spectral radius of the associated
\emph{Markov matrix}.
Let us recall such result.

Let $W$ be a Markov partition of a map {\map{\phi}{\SI}}, and let
$U_1, U_2, \dots,\linebreak[1] U_{\abs{W}}$ be a labelling of the $W-$basic
intervals.
The \emph{Markov matrix of $W$} is defined as the
$\abs{W} \times \abs{W}$ $(0,1)-$matrix
$M=(m_{ij})_{i,j=1}^{\abs{W}}$
such that
$m_{ij}=1$ if and only if $\phi(U_i) \supset U_j$.

For any square matrix $M$, we will denote its \emph{spectral radius}
by $\rho(M)$.

It is well known
(see for instance \cite{BGMY} or \cite[Theorem~4.4.5]{ALM}),
that if $\phi$  is monotone on each basic interval then
\begin{equation}\label{monotoneTopEnt}
  \htop(\phi) = \log\max\{\rho(M),1\}.
\end{equation}

We will use \eqref{monotoneTopEnt} to compute $\htop(\Phi_{P_n^{^+}}).$
To this end we first have to compute the Markov matrix of $\cS$
that, in what follows, will be denoted by $M_n^{^+}.$ As we will
see, a direct computation of $\rho\big(M_n^{^+}\big)$ is
infeasible at a practical level because the size of the matrix grows
quadratically with $n$. So, the computation of
$\rho\big(M_n^{^+}\big)$ will be done in two steps by using
spectral radius preserving transformations of the matrix
$M_n^{^+}$. In this section we will compute the Markov matrix
$M_n^{^+}$ for a symmetric presentation in the orientable case.

To do this, we need to specify completely the map $\Phi_{P_n^{^+}}$
and then compute its Markov matrix.
Recall that, for the symmetric presentation $P_n^{^+}$
(see comments after Proposition~\ref{old-prop:2.1}),
the cyclic ordering of Lemma~\ref{old-lem:2.2} at any vertex is given by:
\[
x_1 < x_2^{-1} < \dots < x_{n-1} < x_{n}^{-1}< x_{1}^{-1} < x_{2} < \dots < x_{n-1}^{-1}<  x_{n} < x_1
\]
(see  Figure~\ref{fig:notationCicOrdSimPres}). The main property of
this cyclic ordering which makes the symmetric presentation very
special and useful is that the edge that is opposite to $x$ at
any vertex is simply the edge $x^{-1}$.

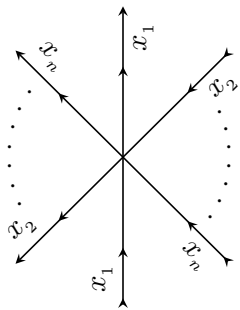
\begin{figure}
\begin{tikzpicture}
\tikzstyle{post}=[>=stealth,semithick]
\tikzset{->-/.style={decoration={markings,mark=at position #1 with {\arrow{>}}},postaction={decorate}}}

\draw[->,post,->-=0.6] (0,0) to node[pos=0.8, below, rotate=90] {$x_{_{1}}$} (90:2);
\draw[>-,post,->-=0.4] (270:2) to node[pos=0.2, above, rotate=(90)] {$x_{_{1}}$} (0,0);

\foreach \angle / \lab in { 225/2, 135/n }{
   \draw[->,post,->-=0.6] (0,0) to node[pos=0.8, above, rotate=\angle-180] {$x_{_{\lab}}$} (\angle:2);
}
\foreach \angle in {325,335,...,385} { \node at (\angle:1.4) {$\cdot$}; \node at (\angle+180:1.5) {$\cdot$}; }
\foreach \angle / \lab in { 45/2, 315/n }{
   \draw[>-,post,->-=0.4] (\angle:2) to node[pos=0.2, below, rotate=(\angle)] {$x_{_{\lab}}$} (0,0);
}
\end{tikzpicture}
\caption{The cyclic ordering of the symmetric presentation.}\label{fig:notationCicOrdSimPres}
\end{figure}

The above cyclic ordering of the generators induces the following ordering of the
intervals $I_x$ along the boundary $\partial\Gamma = \SI$:
\[
 I_{x_1} < I_{x_2^{-1}} < \dots < I_{x_{n}^{-1}} <
 I_{x_{1}^{-1}} < I_{x_2} < \dots < I_{x_{n}} < I_{x_1}.
\]

The fact that the symmetric presentation has associated the above
cyclic ordering gives the following immediate corollary of
Lemma~\ref{old-lem:2.10}:

\begin{corollary}\label{CentralImage}
Let $P_n^{^+}$ be the symmetric presentation of an orientable surface group of rank $n$. Then,
$\Phi_{P_n^{^+}}(C_x) = I_x$ for each generator $x$.
\end{corollary}

For notational reasons we denote the ordered generators as
\[
  y_1 < y_2 < \dots  < y_{2n} < y_1
\]
and the corresponding intervals as
\[
  I_{y_1} < I_{y_2} < \dots  < I_{y_{2n}} < I_{y_1},
\]
where $y_i = x_i^{(-1)^{i+1}}$ for $1\le i\le n$, and $y_i =
x_{i-n}^{(-1)^i}$ for $n+1\le i\le 2n.$ Also, the fact that the
edge that is opposite to $x$ at any vertex is the edge $x^{-1}$ now
gives
\begin{equation}\label{eq:19}
y_i^{-1} = y_{\sbmod{i+n}{2n}}.
\end{equation}

Observe from  ~\eqref{eq:PartInt}  that each of the $2n$ intervals $I_{y_i}$
is divided into $2n-1$ intervals
\begin{equation}\label{BlockOrdering}
L_{y_i}^{n} < \dots < L_{y_i}^{3} <
C_{y_i}^L < C_{y_i} < C_{y_i}^R <
R_{y_i}^{3} < \dots < R_{y_i}^{n} .
\end{equation}
Hence, $\abs{\cS} = 2n(2n-1)$ and thus, the matrix $M_n^{^+}$ is
$2n(2n-1) \times 2n(2n-1).$

Equations~\eqref{old-eq:18}, \eqref{eq:19}
and Corollary~\ref{CentralImage} give the following images of
the partition intervals defined in \eqref{eq:PartInt}
(see Figure~\ref{old-fig:4}):
\begin{equation}\label{eq:TheImage}
\begin{split}
\Phi_{P_n^{^+}}(L_{y_i}^{j}) &= L_{y_{\sbmod{i+n+1}{2n}}}^{j-1}
    \text{for $j \in \{4, 5,\dots, n\}$},\\
\Phi_{P_n^{^+}}(L_{y_i}^{3}) &= C_{y_{\sbmod{i+n+1}{2n}}}^{L} \cup
          C_{y_{\sbmod{i+n+1}{2n}}},\\
\Phi_{P_n^{^+}}(C_{y_i}^{L}) &=
   C_{y_{\sbmod{i+n+1}{2n}}}^{R} \cup
   \left( \bigcup_{j=2}^{n} R_{y_{\sbmod{i+n+1}{2n}}}^{j}\right) \cup
   \left( \bigcup_{k=\sbmod{i+n+2}{2n}}^{\sbmod{i-1}{2n}} I_{y_{k}}\right),\\
\Phi_{P_n^{^+}}(C_{y_i}) &= I_{y_i},\\
\Phi_{P_n^{^+}}(C_{y_i}^{R}) &=
   C_{y_{\sbmod{i+n-1}{2n}}}^{L} \cup
   \left( \bigcup_{j=2}^{n} L_{y_{\sbmod{i+n-1}{2n}}}^{j}\right) \cup
   \left( \bigcup_{k=\sbmod{i+1}{2n}}^{\sbmod{i+n-2}{2n}} I_{y_{k}}\right),\\
\Phi_{P_n^{^+}}(R_{y_i}^{3}) &=  C_{y_{\sbmod{i+n-1}{2n}}} \cup
          C_{y_{\sbmod{i+n-1}{2n}}}^{R},\\
\Phi_{P_n^{^+}}(R_{y_i}^{j}) &= R_{y_{\sbmod{i+n-1}{2n}}}^{j-1}
    \text{for $j \in \{4, 5,\dots, n\}$}.
\end{split}
\end{equation}
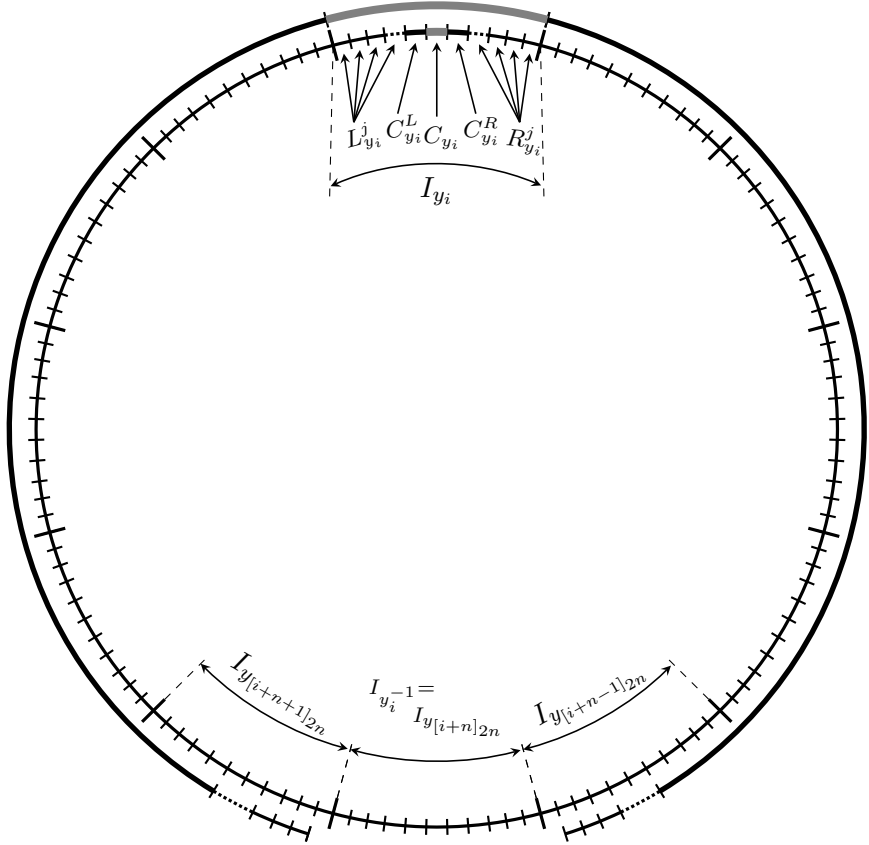
\begin{figure}
\tikzstyle{post}=[->,shorten >=1pt,>=stealth,semithick]
\begin{tikzpicture}
\foreach \angle in {15,45,105,135,...,360}{
    \foreach \ini / \fin in { 0/2.5, 2.5/5, 5/7.5,
                              7.5/10.5, 10.5/13.5,
                              13.5/16.5,
                              16.5/19.5, 19.5/22.5,
                              22.5/25, 25/27.5, 27.5/30}{
        \draw[very thick] (\angle + \ini:150pt) arc (\angle + \ini:\angle + \fin:150pt);
        \draw[thick] (\angle + \ini:147pt) -- +(\angle + \ini:6pt);
    }
    \draw[very thick] (\angle:144pt) -- +(\angle:12pt);
}
\foreach \ini in { 0,2.5,5,7.5,10.5,13.5,16.5,19.5,22.5,25,27.5,30}{ \draw[thick] (75 + \ini:147pt) -- +(75 + \ini:6pt); }
\draw[very thick] (75:144pt) -- +(75:12pt);
\draw[very thick] (75:150pt) arc (75:82.5:150pt); \draw[very thick] (97.5:150pt) arc (97.5:105:150pt);
\draw[very thick, style=densely dotted] (82.5:150pt) arc (82.5:85.5:150pt); \draw[very thick, style=densely dotted] (94.5:150pt) arc (94.5:97.5:150pt);
\draw[line width=2pt] (85.5:150pt) arc (85.5:88.5:150pt); \draw[line width=2pt] (91.5:150pt) arc (91.5:94.5:150pt);
\draw[line width=3pt, color=black!50] (88.5:150pt) arc (88.5:91.5:150pt);

\foreach \ini in { 225, 255, 285 }{
   \draw[dashed] (\ini:150pt) -- (\ini:120pt);
   \draw[dashed] (\ini + 30:150pt) -- (\ini + 30:120pt);
   \draw[<->,shorten >=1pt,>=stealth,semithick] (\ini:125pt) arc (\ini:\ini + 30:125pt);
}
\node[rotate=330] at (240:115pt) [anchor=center] {$I_{y_{\sbmod{i+n+1}{2n}}}$};
\node at (270:105pt) [anchor=center] {$\begin{subarray}{l}I_{y^{-1}_{i}} =\\[-3pt]
                           \hspace*{1.5em}I_{y_{\sbmod{i+n}{2n}}}\end{subarray}$};
\node[rotate=30]  at (300:115pt) [anchor=center] {$I_{y_{\sbmod{i+n-1}{2n}}}$};

\node[rotate=15,anchor=north] at (104:123pt) {\footnotesize $L^j_{y_i}$};
\draw[post] (105:120pt) -- (103.75:147pt);
\draw[post] (105:120pt) -- (101.25:147pt);
\draw[post] (105:120pt) -- (98.75:147pt);
\draw[post] (105:120pt) -- (96:147pt);

\node[rotate=5,anchor=north] at (96:123pt) {\footnotesize $C^L_{y_i}$}; \draw[post] (97:120pt) -- (93:147pt);
\node[anchor=north] at (89:118pt) {\footnotesize $C_{y_i}$}; \draw[post] (90:118pt) -- (90:147pt);
\node[rotate=359,anchor=north] at (82:123pt) {\footnotesize $C^R_{y_i}$}; \draw[post] (83:120pt) -- (87:147pt);

\node[rotate=355,anchor=north] at (74:123pt) {\footnotesize $R^j_{y_i}$};
\draw[post] (75:120pt) -- (76.25:147pt);
\draw[post] (75:120pt) -- (78.75:147pt);
\draw[post] (75:120pt) -- (81.25:147pt);
\draw[post] (75:120pt) -- (84:147pt);

\draw[dashed] (105:150pt) -- (115:95pt); \draw[dashed] (75:150pt) -- (65:95pt);
\draw[<->,shorten >=1pt,>=stealth,semithick] (67:100pt)
                arc (67:114:100pt)
                node at (90:90pt) [anchor=center] {$I_{y_i}$};

\foreach \ini in {287.5,290,292.5,295.5,301.5,435,75,105,238.5,244.5,247.5,250,252.5}{ \draw[thick] (\ini:157pt) -- +(\ini:6pt); }
\draw[thick] (252.5:157pt) -- +(252.5:6pt);

\draw[line width=3pt, color=black!50] (75:160pt) arc (75:105:160pt);
\draw[line width=2pt] (301.5:160pt) arc (301.5:435:160pt);\draw[line width=2pt] (105:160pt) arc (105:238.5:160pt);
\draw[very thick, style=densely dotted] (295.5:160pt) arc (295.5:301.5:160pt);
      \draw[very thick, style=densely dotted] (238.5:160pt) arc (238.5:244.5:160pt);
\draw[very thick] (287.5:160pt) arc (287.5:295.5:160pt);\draw[very thick] (244.5:160pt) arc (244.5:252.5:160pt);
\end{tikzpicture}
\caption{The intervals $I_{y_i}$ in the circle together with the
interior intervals.
The outer curve is the image $\Phi_{P_n^{^+}}\evalat{I_{y_i}}$
(which is order preserving).
The intervals $L^j_{y_i}$, $R^j_{y_i}$ and their images are drawn with
a continuous black line,
$L^3_{y_i}$, $R^3_{y_i}$ and their images are drawn with a dotted line,
$C^L_{y_i}$, $C^R_{y_i}$ and their images are drawn with a continuous
thick black line and finally,
$C_{y_i}$ and its image are drawn with a continuous thick grey
line.}\label{old-fig:4}
\end{figure}

From the formulae \eqref{eq:TheImage} it follows that the Markov
matrix $M_n^{^+}$ has a structure in blocks,
all of size $(2n-1) \times (2n-1).$
So, it is convenient to write the matrix $M_n^{^+}$ as
\begin{equation}\label{MarkovMatrix}
\begin{pmatrix}
M_{11} & M_{12} & \dots & M_{1,2n}\\
M_{21} & M_{22} & \dots & M_{2,2n}\\
\hdotsfor{4}\\
M_{n1} & M_{n2} & \dots & M_{n,2n}\\
\hdotsfor{4}\\
M_{2n,1} & M_{2n,2} & \dots & M_{2n,2n}\\
\end{pmatrix}
\end{equation}
where each of the matrices $M_{lt} = (m^{lt}_{ij})_{i,j=1}^{2n-1}$
is of size $(2n-1) \times (2n-1).$

Accordingly, we will label the basic intervals contained in $I_{y_i}$
as $U^i_j$ in such a way that they preserve the ordering given in
\eqref{BlockOrdering}.
So, for $i = 1,2,\dots, 2n$ and $j=1,2,\dots,2n-1$ it follows that
\begin{equation}\label{MarkovLabelling}
U^i_j = \begin{cases}
   L_{y_i}^{(n + 1) - j} & \text{for $j=1,2,\dots,n-2$},\\
   C_{y_i}^{L} & \text{for $j=n-1$},\\
   C_{y_i} & \text{for $j=n$},\\
   C_{y_i}^{R} & \text{for $j=n+1$},\\
   R_{y_i}^{j - (n - 1)} & \text{for $j=n+2,n+3,\dots,2n-1$}.
\end{cases}
\end{equation}
With this labelling we define the matrix $M_n^{^+}$ so that
$m^{lt}_{ij} = 1$ if and only if $\phi(U^l_i) \supset U^t_j$.

The next theorem is a first reduction in the effective computation of
$\htop(\Phi_{P_n^{^+}})$.

\begin{theorem}\label{firstreduction}
\[
  \htop(\Phi_{P_n^{^+}}) =
  \log\max\left\{\rho\big(M_n^{^+}\big),1\right\} =
  \log\max\left\{\rho\left(\sum_{k=1}^{2n} M_{1k}\right),1\right\}.
\]
\end{theorem}

An \emph{$(r,s)-$block circulant matrix} is a matrix of the form
\[
\begin{pmatrix}
A_1 & A_2 & A_3 & \dots & A_r\\
A_r & A_1 & A_2 & \dots & A_{r-1}\\
A_{r-1} & A_r & A_1 & \dots & A_{r-2}\\
\hdotsfor{5}\\
A_2  & A_3 & A_4 & \dots & A_1
\end{pmatrix}
\]
where each $A_i$ is an $s \times s$ matrix.
Notice that a circulant matrix is completely determined by its first
block row $(A_1\ A_2\ A_3\ \dots\ A_r)$.

The next lemma will be crucial in effectively computing the spectral
radius of $M_n^{^+}$ (see Figure~\ref{BigMatg2} for an example
in the case of rank 4).
\begin{figure}
\fontsize{10}{12}
\def\I{1}
\def\Z{\textcolor{gray}{0}}
\[
\setcounter{MaxMatrixCols}{56}\arraycolsep=0.5pt\def\arraystretch{0.5}
\left(\begin{array}{ccccccc|ccccccc|ccccccc|ccccccc|ccccccc|ccccccc|ccccccc|ccccccc}
\Z&\Z&\Z&\Z&\Z&\Z&\Z&\Z&\Z&\Z&\Z&\Z&\Z&\Z&\Z&\Z&\Z&\Z&\Z&\Z&\Z&\Z&\Z&\Z&\Z&\Z&\Z&\Z&\Z&\Z&\Z&\Z&\Z&\Z&\Z&\Z&\I&\Z&\Z&\Z&\Z&\Z&\Z&\Z&\Z&\Z&\Z&\Z&\Z&\Z&\Z&\Z&\Z&\Z&\Z&\Z\\
\Z&\Z&\Z&\Z&\Z&\Z&\Z&\Z&\Z&\Z&\Z&\Z&\Z&\Z&\Z&\Z&\Z&\Z&\Z&\Z&\Z&\Z&\Z&\Z&\Z&\Z&\Z&\Z&\Z&\Z&\Z&\Z&\Z&\Z&\Z&\Z&\Z&\I&\I&\Z&\Z&\Z&\Z&\Z&\Z&\Z&\Z&\Z&\Z&\Z&\Z&\Z&\Z&\Z&\Z&\Z\\
\Z&\Z&\Z&\Z&\Z&\Z&\Z&\Z&\Z&\Z&\Z&\Z&\Z&\Z&\Z&\Z&\Z&\Z&\Z&\Z&\Z&\Z&\Z&\Z&\Z&\Z&\Z&\Z&\Z&\Z&\Z&\Z&\Z&\Z&\Z&\Z&\Z&\Z&\Z&\I&\I&\I&\I&\I&\I&\I&\I&\I&\I&\I&\I&\I&\I&\I&\I&\I\\
\I&\I&\I&\I&\I&\I&\I&\Z&\Z&\Z&\Z&\Z&\Z&\Z&\Z&\Z&\Z&\Z&\Z&\Z&\Z&\Z&\Z&\Z&\Z&\Z&\Z&\Z&\Z&\Z&\Z&\Z&\Z&\Z&\Z&\Z&\Z&\Z&\Z&\Z&\Z&\Z&\Z&\Z&\Z&\Z&\Z&\Z&\Z&\Z&\Z&\Z&\Z&\Z&\Z&\Z\\
\Z&\Z&\Z&\Z&\Z&\Z&\Z&\I&\I&\I&\I&\I&\I&\I&\I&\I&\I&\I&\I&\I&\I&\I&\I&\I&\Z&\Z&\Z&\Z&\Z&\Z&\Z&\Z&\Z&\Z&\Z&\Z&\Z&\Z&\Z&\Z&\Z&\Z&\Z&\Z&\Z&\Z&\Z&\Z&\Z&\Z&\Z&\Z&\Z&\Z&\Z&\Z\\
\Z&\Z&\Z&\Z&\Z&\Z&\Z&\Z&\Z&\Z&\Z&\Z&\Z&\Z&\Z&\Z&\Z&\Z&\Z&\Z&\Z&\Z&\Z&\Z&\I&\I&\Z&\Z&\Z&\Z&\Z&\Z&\Z&\Z&\Z&\Z&\Z&\Z&\Z&\Z&\Z&\Z&\Z&\Z&\Z&\Z&\Z&\Z&\Z&\Z&\Z&\Z&\Z&\Z&\Z&\Z\\
\Z&\Z&\Z&\Z&\Z&\Z&\Z&\Z&\Z&\Z&\Z&\Z&\Z&\Z&\Z&\Z&\Z&\Z&\Z&\Z&\Z&\Z&\Z&\Z&\Z&\Z&\I&\Z&\Z&\Z&\Z&\Z&\Z&\Z&\Z&\Z&\Z&\Z&\Z&\Z&\Z&\Z&\Z&\Z&\Z&\Z&\Z&\Z&\Z&\Z&\Z&\Z&\Z&\Z&\Z&\Z\\
\hline\rule{0pt}{8pt}
\Z&\Z&\Z&\Z&\Z&\Z&\Z&\Z&\Z&\Z&\Z&\Z&\Z&\Z&\Z&\Z&\Z&\Z&\Z&\Z&\Z&\Z&\Z&\Z&\Z&\Z&\Z&\Z&\Z&\Z&\Z&\Z&\Z&\Z&\Z&\Z&\Z&\Z&\Z&\Z&\Z&\Z&\Z&\I&\Z&\Z&\Z&\Z&\Z&\Z&\Z&\Z&\Z&\Z&\Z&\Z\\
\Z&\Z&\Z&\Z&\Z&\Z&\Z&\Z&\Z&\Z&\Z&\Z&\Z&\Z&\Z&\Z&\Z&\Z&\Z&\Z&\Z&\Z&\Z&\Z&\Z&\Z&\Z&\Z&\Z&\Z&\Z&\Z&\Z&\Z&\Z&\Z&\Z&\Z&\Z&\Z&\Z&\Z&\Z&\Z&\I&\I&\Z&\Z&\Z&\Z&\Z&\Z&\Z&\Z&\Z&\Z\\
\I&\I&\I&\I&\I&\I&\I&\Z&\Z&\Z&\Z&\Z&\Z&\Z&\Z&\Z&\Z&\Z&\Z&\Z&\Z&\Z&\Z&\Z&\Z&\Z&\Z&\Z&\Z&\Z&\Z&\Z&\Z&\Z&\Z&\Z&\Z&\Z&\Z&\Z&\Z&\Z&\Z&\Z&\Z&\Z&\I&\I&\I&\I&\I&\I&\I&\I&\I&\I\\
\Z&\Z&\Z&\Z&\Z&\Z&\Z&\I&\I&\I&\I&\I&\I&\I&\Z&\Z&\Z&\Z&\Z&\Z&\Z&\Z&\Z&\Z&\Z&\Z&\Z&\Z&\Z&\Z&\Z&\Z&\Z&\Z&\Z&\Z&\Z&\Z&\Z&\Z&\Z&\Z&\Z&\Z&\Z&\Z&\Z&\Z&\Z&\Z&\Z&\Z&\Z&\Z&\Z&\Z\\
\Z&\Z&\Z&\Z&\Z&\Z&\Z&\Z&\Z&\Z&\Z&\Z&\Z&\Z&\I&\I&\I&\I&\I&\I&\I&\I&\I&\I&\I&\I&\I&\I&\I&\I&\I&\Z&\Z&\Z&\Z&\Z&\Z&\Z&\Z&\Z&\Z&\Z&\Z&\Z&\Z&\Z&\Z&\Z&\Z&\Z&\Z&\Z&\Z&\Z&\Z&\Z\\
\Z&\Z&\Z&\Z&\Z&\Z&\Z&\Z&\Z&\Z&\Z&\Z&\Z&\Z&\Z&\Z&\Z&\Z&\Z&\Z&\Z&\Z&\Z&\Z&\Z&\Z&\Z&\Z&\Z&\Z&\Z&\I&\I&\Z&\Z&\Z&\Z&\Z&\Z&\Z&\Z&\Z&\Z&\Z&\Z&\Z&\Z&\Z&\Z&\Z&\Z&\Z&\Z&\Z&\Z&\Z\\
\Z&\Z&\Z&\Z&\Z&\Z&\Z&\Z&\Z&\Z&\Z&\Z&\Z&\Z&\Z&\Z&\Z&\Z&\Z&\Z&\Z&\Z&\Z&\Z&\Z&\Z&\Z&\Z&\Z&\Z&\Z&\Z&\Z&\I&\Z&\Z&\Z&\Z&\Z&\Z&\Z&\Z&\Z&\Z&\Z&\Z&\Z&\Z&\Z&\Z&\Z&\Z&\Z&\Z&\Z&\Z\\
\hline\rule{0pt}{8pt}
\Z&\Z&\Z&\Z&\Z&\Z&\Z&\Z&\Z&\Z&\Z&\Z&\Z&\Z&\Z&\Z&\Z&\Z&\Z&\Z&\Z&\Z&\Z&\Z&\Z&\Z&\Z&\Z&\Z&\Z&\Z&\Z&\Z&\Z&\Z&\Z&\Z&\Z&\Z&\Z&\Z&\Z&\Z&\Z&\Z&\Z&\Z&\Z&\Z&\Z&\I&\Z&\Z&\Z&\Z&\Z\\
\Z&\Z&\Z&\Z&\Z&\Z&\Z&\Z&\Z&\Z&\Z&\Z&\Z&\Z&\Z&\Z&\Z&\Z&\Z&\Z&\Z&\Z&\Z&\Z&\Z&\Z&\Z&\Z&\Z&\Z&\Z&\Z&\Z&\Z&\Z&\Z&\Z&\Z&\Z&\Z&\Z&\Z&\Z&\Z&\Z&\Z&\Z&\Z&\Z&\Z&\Z&\I&\I&\Z&\Z&\Z\\
\I&\I&\I&\I&\I&\I&\I&\I&\I&\I&\I&\I&\I&\I&\Z&\Z&\Z&\Z&\Z&\Z&\Z&\Z&\Z&\Z&\Z&\Z&\Z&\Z&\Z&\Z&\Z&\Z&\Z&\Z&\Z&\Z&\Z&\Z&\Z&\Z&\Z&\Z&\Z&\Z&\Z&\Z&\Z&\Z&\Z&\Z&\Z&\Z&\Z&\I&\I&\I\\
\Z&\Z&\Z&\Z&\Z&\Z&\Z&\Z&\Z&\Z&\Z&\Z&\Z&\Z&\I&\I&\I&\I&\I&\I&\I&\Z&\Z&\Z&\Z&\Z&\Z&\Z&\Z&\Z&\Z&\Z&\Z&\Z&\Z&\Z&\Z&\Z&\Z&\Z&\Z&\Z&\Z&\Z&\Z&\Z&\Z&\Z&\Z&\Z&\Z&\Z&\Z&\Z&\Z&\Z\\
\Z&\Z&\Z&\Z&\Z&\Z&\Z&\Z&\Z&\Z&\Z&\Z&\Z&\Z&\Z&\Z&\Z&\Z&\Z&\Z&\Z&\I&\I&\I&\I&\I&\I&\I&\I&\I&\I&\I&\I&\I&\I&\I&\I&\I&\Z&\Z&\Z&\Z&\Z&\Z&\Z&\Z&\Z&\Z&\Z&\Z&\Z&\Z&\Z&\Z&\Z&\Z\\
\Z&\Z&\Z&\Z&\Z&\Z&\Z&\Z&\Z&\Z&\Z&\Z&\Z&\Z&\Z&\Z&\Z&\Z&\Z&\Z&\Z&\Z&\Z&\Z&\Z&\Z&\Z&\Z&\Z&\Z&\Z&\Z&\Z&\Z&\Z&\Z&\Z&\Z&\I&\I&\Z&\Z&\Z&\Z&\Z&\Z&\Z&\Z&\Z&\Z&\Z&\Z&\Z&\Z&\Z&\Z\\
\Z&\Z&\Z&\Z&\Z&\Z&\Z&\Z&\Z&\Z&\Z&\Z&\Z&\Z&\Z&\Z&\Z&\Z&\Z&\Z&\Z&\Z&\Z&\Z&\Z&\Z&\Z&\Z&\Z&\Z&\Z&\Z&\Z&\Z&\Z&\Z&\Z&\Z&\Z&\Z&\I&\Z&\Z&\Z&\Z&\Z&\Z&\Z&\Z&\Z&\Z&\Z&\Z&\Z&\Z&\Z
\end{array}\right)
\]
\caption{The first three (of the total of eight) block rows of the Markov matrix
$M_{P_4^{^+}}$ corresponding to the symmetric presentation of an orientable surface group of rank 4.}\label{BigMatg2}
\end{figure}

\begin{lemma}\label{Markovcirculating}
The Markov matrix $M_n^{^+}$ is a $(2n, 2n-1)-$block circulant
matrix.
\end{lemma}

\begin{proof}
From the formulae \eqref{eq:TheImage} it follows that
$\Phi_{P_n^{^+}}(U^l_i) \supset U^t_j$
if and only if
$\Phi_{P_n^{^+}}(U^{\sbmod{l+1}{2n}}_i) \supset U^{\sbmod{t+1}{2n}}_j.$
In terms of the Markov matrix this amounts to
$m^{lt}_{ij} = m^{\sbmod{l+1}{2n},\sbmod{t+1}{2n}}_{ij}$
for every
$l,t \in \{1,2,\dots,2n\}$ and
$i,j \in \{1,2,\dots,2n-1\}.$
This implies that
$M_{lt} = M_{\sbmod{l+1}{2n},\sbmod{t+1}{2n}}$.
\end{proof}

The next technical lemma provides a nice and useful result about the
spectral radius of block circulant matrices.

\begin{lemma}\label{Spectralcirculating}
Let
\[
A = \begin{pmatrix}
A_1 & A_2 & A_3 & \dots & A_r\\
A_r & A_1 & A_2 & \dots & A_{r-1}\\
A_{r-1} & A_r & A_1 & \dots & A_{r-2}\\
\hdotsfor{5}\\
A_2  & A_3 & A_4 & \dots & A_1
\end{pmatrix}
\]
be a non-negative block circulant matrix.
Then
\[
 \rho(A) = \rho\left(\sum_{i=1}^r A_i \right).
\]
\end{lemma}

\begin{proof}
Since $A$ is a block matrix, for every $m \ge 1$,
$A^m$ is a block matrix
\[
\begin{pmatrix}
A^{(m)}_{11} & A^{(m)}_{12} & \dots & A^{(m)}_{1r}\\
A^{(m)}_{21} & A^{(m)}_{22} & \dots & A^{(m)}_{2r}\\
\hdotsfor{4}\\
A^{(m)}_{r1} & A^{(m)}_{r2} & \dots & A^{(m)}_{rr}
\end{pmatrix}
\]
where each block has size $s \times s$ and is a sum of
$r^{m-1}$ non-commutative products of $m$ matrices among the blocks
$A_1,A_2,A_3,\dots,A_r.$
That is, each $A^{(m)}_{ij}$ is the sum of $r^{m-1}$ matricial
products of the form $A_{r_1}A_{r_2}\cdots A_{r_m}$ with $r_1,
r_2,\dots, r_m \in \{1,2,\dots,r\}.$

Moreover, since every block $A_i$ appears exactly once in every block
row and every block column it can be proved by induction that every
product of the form $A_{r_1}A_{r_2}\cdots A_{r_m}$ appears exactly
once in every block row and every block column of $A^m.$
Therefore, for every $q \in \{1,2,\dots,r\},$
$\sum_{i=1}^r A^{(m)}_{qi}$ and $\sum_{i=1}^r A^{(m)}_{iq}$ are the
sum of $r^m$ matricial products, and every product
$A_{r_1}A_{r_2}\cdots A_{r_m}$ appears exactly once in each of these
expressions. Hence,
\[
\sum_{i=1}^r A^{(m)}_{qi} =
   \sum_{i=1}^r A^{(m)}_{iq} =
   \left(\sum_{i=1}^r A_i \right)^m .
\]
Consequently, the classical matrix norms satisfy:
\begin{center}
$
\norm{A^m}_\infty = \norm{\left(\sum_{i=1}^r A_i \right)^m}_\infty
$
and
$
\norm{A^m}_1 = \norm{\left(\sum_{i=1}^r A_i \right)^m}_1.
$
\end{center}

Since the spectral radius is defined as: $\rho(A) = \lim_{m\to\infty} \norm{A^m}^{1/m}$ for any matrix
norm, it follows that $\rho(A) = \rho\left(\sum_{i=1}^r A_i \right).$
\end{proof}

\begin{remark}
In fact the above lemma holds for every matrix for which a given
block appears exactly once in every block row and every block column.
\end{remark}

Now we are endowed with the necessary ingredients to prove Theorem~\ref{firstreduction}.

\begin{proof}[Proof of Theorem~\ref{firstreduction}]
It follows from \eqref{monotoneTopEnt} and
Lemmas~\ref{Markovcirculating} and \ref{Spectralcirculating}.
\end{proof}

Next, to complete the first reduction in the computation of
$\htop(\Phi_{P_n^{^+}})$, we will give an explicit formula for the
matrix $\sum_{k=1}^{2n} M_{1k}.$

The \emph{compacted matrix of rank $n$} is the $(2n-1) \times
(2n-1)$ matrix $\mathsf{C}_n = (c_{ij})$ defined by
\begin{equation}
c_{ij} = \begin{cases}
1   & \text{$j = i+1$ and $i \in \{1,2,\dots,n-3\}$},\\
1   & \text{if $j \in \{n-1, n\}$ and $i=n-2$},\\
n-2 & \text{if $j \in \{1,2,\dots, n\}$ and $i=n-1$},\\
n-1 & \text{if $j \in \{n+1,n+2,\dots, 2n-1\}$ and $i=n-1$},\\
1   & \text{if $i=n$},\\
n-1 & \text{if $j \in \{1,2,\dots, n-1\}$ and $i=n+1$},\\
n-2 & \text{if $j \in \{n,n+1,\dots, 2n-1\}$ and $i=n+1$},\\
1   & \text{if $j \in \{n, n+1\}$ and $i=n+2$},\\
1   & \text{$j = i-1$ and $i \in \{n+3, n+4, \dots,2n-1\}$, and}\\
0   & \text{otherwise.}
\end{cases}
\end{equation}
In matrix form, $\mathsf{C}_n$ is
{\fontsize{9}{11}
\[
\setcounter{MaxMatrixCols}{20}\arraycolsep=1.5pt\def\arraystretch{0.5}
\begin{pmatrix}\label{CM-MF}
0 & 1 & 0 & 0 & \cdots & 0 & 0 & 0 & 0 & \cdots & 0 & 0 & 0 \\
0 & 0 & 1 & 0 & \cdots & 0 & 0 & 0 & 0 & \cdots & 0 & 0 & 0 \\
\vdots & \vdots & \vdots & \ddots & \cdots & \vdots &
    \vdots & \vdots & \vdots & \vdots & \vdots & \vdots & \vdots \\
\vdots & \vdots & \vdots & \cdots & \ddots & \vdots &
    \vdots & \vdots & \vdots & \vdots & \vdots & \vdots & \vdots \\
0 & 0 & 0 & \cdots & 0 & 1 & 1 & 0 & 0 & \cdots & 0 & 0 & 0 \\
n-2 & n-2  & n-2 & \cdots & n-2 & n-2 & n-2 & n-1 & n-1 & \cdots & n-1 & n-1 & n-1 \\
1 & 1 & 1 & \cdots & 1 & 1 & 1 & 1 & 1 & \cdots & 1 & 1 & 1 \\
n-1 & n-1 & n-1 & \cdots & n-1 & n-1 & n-2 & n-2 & n-2 & \cdots & n-2 & n-2 & n-2 \\
0 & 0 & 0 & \cdots & 0 & 0 & 1 & 1 & 0 & \cdots & 0 & 0 & 0 \\
\vdots & \vdots & \vdots & \vdots & \vdots & \vdots & \vdots &
             \vdots & \ddots & \cdots  & \vdots & \vdots & \vdots \\
\vdots & \vdots & \vdots & \vdots & \vdots & \vdots & \vdots &
             \vdots & \cdots & \ddots  & \vdots & \vdots & \vdots \\
0 & 0 & 0 & 0 & \cdots & 0 & 0 & 0 & 0 & \cdots & 1 & 0 & 0 \\
0 & 0 & 0 & 0 & \cdots & 0 & 0 & 0 & 0 & \cdots & 0 & 1 & 0 \\
\end{pmatrix}
\]}

\begin{lemma}\label{CompressedMatrix}
\[
\sum_{k=1}^{2n} M_{1k} = \mathsf{C}_n.
\]
\end{lemma}

To prove this lemma we need to explicitly describe the matrix
$M_n^{^+}.$ To this end we introduce the following notation.
The zero matrix of size $k \times k$ will be denoted by
$\mathbf{0}_{k},$ and $\mathbf{J}_{k}$ will denote the $k \times k$
$(0,1)-$matrix with ones in the anti-diagonal.
Also, if $i \in \{1,2,\dots,k\},$ $\mathbf{U}^i_{k}$ will denote the
$k \times k$ matrix such that all entries in the $i-$th row are 1
and all other entries are 0.
Finally, for $k \ge 5$ odd, $\mathbf{T}_{k} = (t_{ij})$ is the
$k \times k$ $(0,1)-$matrix such that $t_{ij} = 1$ if and only if (see
Figure~\ref{Blocks} for examples):
\begin{itemize}
 \item $j = i+1$ and $i \in \{1,2,\dots,\widetilde{k}-3\}$, or
 \item $j \in \{\widetilde{k}-1, \widetilde{k}\}$ and $i=\widetilde{k}-2$ or
 \item $\widetilde{k}+1 \le j \le k$ and $i=\widetilde{k}-1,$
\end{itemize}
where $\widetilde{k} = \tfrac{k+1}{2}.$
Observe that (see again Figure~\ref{Blocks})
$\mathbf{J}_{k}\mathbf{T}_{k}\mathbf{J}_{k}$ is the matrix
obtained from $\mathbf{T}_{k}$ by a symmetry with respect to the
central coordinate $t_{\widetilde{k},\widetilde{k}}$.
\begin{figure}
\[
\begin{array}{lr}
\mathbf{U}^3_7 = \begin{pmatrix}
0&0&0&0&0&0&0 \\
0&0&0&0&0&0&0 \\
1&1&1&1&1&1&1 \\
0&0&0&0&0&0&0 \\
0&0&0&0&0&0&0 \\
0&0&0&0&0&0&0 \\
0&0&0&0&0&0&0 \\
\end{pmatrix} & \mathbf{T}_7 = \begin{pmatrix}
0&1&0&0&0&0&0 \\
0&0&1&1&0&0&0 \\
0&0&0&0&1&1&1 \\
0&0&0&0&0&0&0 \\
0&0&0&0&0&0&0 \\
0&0&0&0&0&0&0 \\
0&0&0&0&0&0&0
\end{pmatrix} \\[10ex]
\mathbf{J}_7 = \begin{pmatrix}
0&0&0&0&0&0&1 \\
0&0&0&0&0&1&0 \\
0&0&0&0&1&0&0 \\
0&0&0&1&0&0&0 \\
0&0&1&0&0&0&0 \\
0&1&0&0&0&0&0 \\
1&0&0&0&0&0&0
\end{pmatrix} & \mathbf{J}_7\mathbf{T}_7\mathbf{J}_7 = \begin{pmatrix}
0&0&0&0&0&0&0 \\
0&0&0&0&0&0&0 \\
0&0&0&0&0&0&0 \\
0&0&0&0&0&0&0 \\
1&1&1&0&0&0&0 \\
0&0&0&1&1&0&0 \\
0&0&0&0&0&1&0
\end{pmatrix}
\end{array}\]
\caption{Examples of the matrices
$\mathbf{U}^i_{k}$,
$\mathbf{T}_{k},$
$\mathbf{J}_{k}$ and
$\mathbf{J}_{k}\mathbf{T}_{k}\mathbf{J}_{k}$
with $k = 7$.}\label{Blocks}
\end{figure}

\begin{proof}[Proof of Lemma~\ref{CompressedMatrix}]
From formulae \eqref{eq:TheImage} and taking into account the
labelling of the basic intervals \eqref{MarkovLabelling}
it follows that $m^{lt}_{ij} = 1$ if and only if
(see Figure~\ref{BigMatg2} for an example in the case of rank 4):
{\small\begin{equation*}
\begin{cases}
j = i+1,\ t = \sbmod{l+n+1}{2n} & \text{for $i=1,2,\dots,n-3$},\\[1ex]
j \in \{n-1, n\},\ t = \sbmod{l+n+1}{2n} & \text{for $i=n-2$},\\[1ex]
\left.\text{\Large $
   \begin{subarray}{l}
       n+1 \le j \le 2n-1,\ t = \sbmod{l+n+1}{2n}, \text{ and}\\
       j \in \{1, 2,\dots, 2n-1\},\ \sbmod{l+n+2}{2n} \le t \le
\sbmod{l-1}{2n}
   \end{subarray}$}\right\}                 & \text{for
$i=n-1$},\\[1ex]
j \in \{1,2,\dots,2n-1\},\ t = l & \text{for $i=n$},\\[1ex]
\left.\text{\Large $
   \begin{subarray}{l}
       j \in \{1,2,\dots, n-1\},\ t = \sbmod{l+n-1}{2n}, \text{ and}\\
       j \in \{1, 2,\dots, 2n-1\},\ \sbmod{l+1}{2n} \le t \le
\sbmod{l+n-2}{2n}
   \end{subarray}$}\right\}                 & \text{for
$i=n+1$},\\[1ex]
j \in \{n, n+1\},\ t = \sbmod{l+n-1}{2n} & \text{for $i=n+2$},\\[1ex]
j = i-1,\ t = \sbmod{l+n-1}{2n} & \text{for $i=n+3,\dots,2n-1$}.
\end{cases}
\end{equation*}}
In matrix block form the above formulae become
(see again Figure~\ref{BigMatg2}):
\begin{equation*}
\begin{split}
M_{l,\sbmod{l+n+1}{2n}} &= \mathbf{T}_{2n-1}\\
M_{lt} &= \mathbf{U}^{n-1}_{2n-1} \text{ for $\sbmod{l+n+2}{2n} \le t
\le \sbmod{l-1}{2n}$} \\
M_{ll} &= \mathbf{U}^{n}_{2n-1} \\
M_{lt} &= \mathbf{U}^{n+1}_{2n-1} \text{ for $\sbmod{l+1}{2n} \le t
\le \sbmod{l+n-2}{2n}$} \\
M_{l,\sbmod{l+n-1}{2n}} &=
\mathbf{J}_{2n-1}\mathbf{T}_{2n-1}\mathbf{J}_{2n-1}\\
M_{l,\sbmod{l+n}{2n}} &= \mathbf{0}_{2n-1}
\end{split}
\end{equation*}
for $l \in \{1,2,\dots, 2n\}.$
Consequently,
\begin{align*}
 \sum_{t=1}^n M_{1t} &=
     \mathbf{T}_{2n-1} +
     \mathbf{J}_{2n-1}\mathbf{T}_{2n-1}\mathbf{J}_{2n-1} +
     \mathbf{U}^{n}_{2n-1} +
     (n-2) \left(\mathbf{U}^{n-1}_{2n-1} +
\mathbf{U}^{n+1}_{2n-1}\right)\\
 & = \mathsf{C}_n.
\end{align*}
\end{proof}

The next corollary gives an explicit formula for the entropy in the
orientable case in terms of the spectral radius of a $(2n-1) \times
(2n-1)$ matrix which is a ``compacted'' version of the Markov matrix
$M_n^{^+}$.

\begin{corollary}\label{reductionToCompressed}
\[
 \htop(\Phi_{P_n^{^+}})  =
\log\max\left\{\rho\left(\mathsf{C}_n\right),1\right\}.
\]
\end{corollary}

\begin{proof}
It follows from  Theorem~\ref{firstreduction} and
Lemma~\ref{CompressedMatrix}.
\end{proof}

\begin{remark}\label{TP:C}
Note that the map $\Phi_{P_n^{^+}}$ commutes with a rigid rotation $R$
of period $2n.$  The quotient space obtained by identifying each
orbit of $R$ to a point is a circle. The map induced by
$\Phi_{P_n^{^+}}$ on this quotient space is also a Markov map. The
matrix $\mathsf{C}_n$ is nothing but the Markov matrix of this
induced map (see \eqref{eq:TheImage} and Figure~\ref{old-fig:4}).
\end{remark}

\section{The non-orientable case}\label{sec:disoriented}
We start this section by extending the definition of \emph{symmetric presentation}
(Definition~\ref{SimPres}) to non orientable surface groups.

\begin{definition}\label{SimPres2}
Given a surface group $\Gamma=\pi_1(S)$ of rank $n$,
where $S$ is a non orientable surface, the following presentation of $\Gamma$ will be called
\emph{symmetric} and denoted by $P_n^{^-}$. Its definition depends on the parity of $n$ as
follows.
For $n$ odd, we define $P_n^{^-}$ as
\[
\present{x_1^{\pm 1}, x_2^{\pm 1}, \dots, x_n^{\pm 1} \left/
        x_1 x_2 \cdots x_n
        x_{n-1}x_{n-2}\cdots x_1x_n \right.}
\]
while, for $n$ even, $P_n^{^-}$ is defined as
\[
\present{x_1^{\pm 1}, x_2^{\pm 1}, \dots, x_n^{\pm 1} \left/
        x_1 x_2 \cdots x_n
        x_{n-1}x_{n-2}\cdots x_1x_n^{-1} \right.}.
\]
\end{definition}

Similar arguments to the ones used in the proof of
Proposition~\ref{old-prop:2.1} yield that the symmetric presentation
$P_n^{^-}$ is minimal and geometric.

As in the orientable case, the nomenclature \emph{symmetric} for the
presentation $P_n^{^-}$ accounts for the fact that, at each vertex,
the cyclic ordering of the generators (Lemma \ref{old-lem:2.2})
exhibits the useful property that the edge opposite to $x$ at any
vertex is simply the edge $x^{-1}$. Indeed, one can check that the
ordering of the generators at any vertex is
\[
  x_1 < x_2^{-1} < x_3 < \dots <
  x_{n-1}< x_n^{-1}<x_1^{-1} < x_2 < x_3^{-1} < \dots <
  x_{n-1}^{-1}< x_n
\]
when $n$ is even, and
\[
  x_1 < x_2^{-1} < x_3 < \dots <
  x_{n-1}^{-1}< x_n < x_1^{-1} < x_2 < x_3^{-1} < \dots <
  x_{n-1}< x_n^{-1}
\]
when $n$ is odd.

The fact that the symmetric presentation has associated the above
cyclic ordering implies that Corollary~\ref{CentralImage} also
holds for the non-orientable case:

\begin{corollary}\label{CentralImage-} Let $P_n^{^-}$ be the symmetric
presentation of a non-orientable surface
group of rank $n.$ Then, $\Phi_{P_n^{^-}}(C_x)=I_x$ for each generator $x$.
\end{corollary}

Notice that map $\Phi_{P_n^{^+}}$ and the Markov matrix $M^+_n$ are
only defined for $n$ even since the group corresponds to an
orientable surface. However, all associated formulae extend to the
case $n$ odd. In this sense below we will compare the maps
$\Phi_{P_n^{^+}}$ and $\Phi_{P_n^{^-}}$ and the associated Markov
matrices $M^+_n$ and $M^-_n,$ independently on the parity of $n.$

Using the notations introduced in the previous sections one can
check that the Markov map $\Phi_{P_n^{^-}}$ behaves essentially as
$\Phi_{P_n^{^+}}$ in all intervals $I_{y_i}$ except when $i \in \{n,
2n\}.$ In these two intervals the map reverses orientation.
So, when $i\notin\{n,2n\}$ the equation~\eqref{eq:TheImage} holds
with $\Phi_{P_n^{^-}}$ instead of $\Phi_{P_n^{^+}}$.
When $i = n,$
\begin{equation}\label{eq:TheImage-}
\begin{split}
\Phi_{P_n^-}(L_{y_n}^{j}) &=     R_{y_{2n-1}}^{j-1}
    \text{for $j \in \{4, 5,\dots, n\}$},\\
\Phi_{P_n^-}(L_{y_n}^{3}) &= C_{y_{2n-1}} \cup
          C_{y_{2n-1}}^{R}, \\
\Phi_{P_n^{^-}}(C_{y_n}^{L}) &=C_{y_{2n-1}}^{L} \cup
   \left( \bigcup_{j=2}^{n} L_{y_{2n-1}}^{j}\right) \cup
   \left( \bigcup_{k=n+1}^{2n-2} I_{y_{k}}\right),\\
\Phi_{P_n^{^-}}(C_{y_n}) &= I_{y_n},\\
\Phi_{P_n^{^-}}(C_{y_n}^{R}) &=
   C_{y_{1}}^{R} \cup
   \left( \bigcup_{j=2}^{n} R_{y_{1}}^{j}\right) \cup
   \left( \bigcup_{k=2}^{n-1} I_{y_{k}}\right),\\
\Phi_{P_n^{^-}}(R_{y_n}^{3}) &=     C_{y_{1}}^{L} \cup
          C_{y_{1}},\\
\Phi_{P_n^{^-}}(R_{y_n}^{j}) &=L_{y_{1}}^{j-1}
    \text{for $j \in \{4, 5,\dots, n\}$},
\end{split}
\end{equation}
and analogous formulae hold for the interval $I_{y_{2n}}$.
Hence, in a similar way to the previous section it follows that the
Markov matrix $M^-_n$ of $\Phi_{P_n^{^-}}$ is of the form
\eqref{MarkovMatrix} with $M_{i,j}$ replaced by
$\mathbf{J}_{2n-1} M_{i,j}$ for $i\in \{n,2n\}$ and $j =
1,2,\dots,2n$ (see Figure~\ref{BigMatrank3}).
\begin{figure}
\fontsize{11}{13}
\def\I{1}
\def\Z{\textcolor{gray}{0}}
\[
\setcounter{MaxMatrixCols}{30}\arraycolsep=1.5pt\def\arraystretch{0.7}
\left(\begin{array}{ccccc|ccccc|ccccc|ccccc|ccccc|ccccc}
\Z&\Z&\Z&\Z&\Z&\Z&\Z&\Z&\Z&\Z&\Z&\Z&\Z&\Z&\Z&\Z&\Z&\Z&\Z&\Z&\Z&\I&\I&\Z&\Z&\Z&\Z&\Z&\Z&\Z\\
\Z&\Z&\Z&\Z&\Z&\Z&\Z&\Z&\Z&\Z&\Z&\Z&\Z&\Z&\Z&\Z&\Z&\Z&\Z&\Z&\Z&\Z&\Z&\I&\I&\I&\I&\I&\I&\I\\
\I&\I&\I&\I&\I&\Z&\Z&\Z&\Z&\Z&\Z&\Z&\Z&\Z&\Z&\Z&\Z&\Z&\Z&\Z&\Z&\Z&\Z&\Z&\Z&\Z&\Z&\Z&\Z&\Z\\
\Z&\Z&\Z&\Z&\Z&\I&\I&\I&\I&\I&\I&\I&\Z&\Z&\Z&\Z&\Z&\Z&\Z&\Z&\Z&\Z&\Z&\Z&\Z&\Z&\Z&\Z&\Z&\Z\\
\Z&\Z&\Z&\Z&\Z&\Z&\Z&\Z&\Z&\Z&\Z&\Z&\I&\I&\Z&\Z&\Z&\Z&\Z&\Z&\Z&\Z&\Z&\Z&\Z&\Z&\Z&\Z&\Z&\Z\\
\hline\rule{0pt}{8pt}
\Z&\Z&\Z&\Z&\Z&\Z&\Z&\Z&\Z&\Z&\Z&\Z&\Z&\Z&\Z&\Z&\Z&\Z&\Z&\Z&\Z&\Z&\Z&\Z&\Z&\Z&\I&\I&\Z&\Z\\
\I&\I&\I&\I&\I&\Z&\Z&\Z&\Z&\Z&\Z&\Z&\Z&\Z&\Z&\Z&\Z&\Z&\Z&\Z&\Z&\Z&\Z&\Z&\Z&\Z&\Z&\Z&\I&\I\\
\Z&\Z&\Z&\Z&\Z&\I&\I&\I&\I&\I&\Z&\Z&\Z&\Z&\Z&\Z&\Z&\Z&\Z&\Z&\Z&\Z&\Z&\Z&\Z&\Z&\Z&\Z&\Z&\Z\\
\Z&\Z&\Z&\Z&\Z&\Z&\Z&\Z&\Z&\Z&\I&\I&\I&\I&\I&\I&\I&\Z&\Z&\Z&\Z&\Z&\Z&\Z&\Z&\Z&\Z&\Z&\Z&\Z\\
\Z&\Z&\Z&\Z&\Z&\Z&\Z&\Z&\Z&\Z&\Z&\Z&\Z&\Z&\Z&\Z&\Z&\I&\I&\Z&\Z&\Z&\Z&\Z&\Z&\Z&\Z&\Z&\Z&\Z\\
\hline\rule{0pt}{8pt}
\Z&\Z&\Z&\Z&\Z&\Z&\Z&\Z&\Z&\Z&\Z&\Z&\Z&\Z&\Z&\Z&\Z&\Z&\Z&\Z&\Z&\Z&\I&\I&\Z&\Z&\Z&\Z&\Z&\Z\\
\Z&\Z&\Z&\Z&\Z&\Z&\Z&\Z&\Z&\Z&\Z&\Z&\Z&\Z&\Z&\I&\I&\I&\I&\I&\I&\I&\Z&\Z&\Z&\Z&\Z&\Z&\Z&\Z\\
\Z&\Z&\Z&\Z&\Z&\Z&\Z&\Z&\Z&\Z&\I&\I&\I&\I&\I&\Z&\Z&\Z&\Z&\Z&\Z&\Z&\Z&\Z&\Z&\Z&\Z&\Z&\Z&\Z\\
\Z&\Z&\I&\I&\Z&\I&\I&\I&\I&\I&\Z&\Z&\Z&\Z&\Z&\Z&\Z&\Z&\Z&\Z&\Z&\Z&\Z&\Z&\Z&\Z&\Z&\Z&\Z&\Z\\
\I&\I&\Z&\Z&\Z&\Z&\Z&\Z&\Z&\Z&\Z&\Z&\Z&\Z&\Z&\Z&\Z&\Z&\Z&\Z&\Z&\Z&\Z&\Z&\Z&\Z&\Z&\Z&\Z&\Z\\
\end{array}\right)
\]
\caption{The first three (of the total of six) block rows of the Markov matrix
$M_{P_3^{^-}}$ corresponding to the symmetric presentation of a non-orientable surface group of rank 3.}\label{BigMatrank3}
\end{figure}

Next we want to prove Lemma~\ref{lem:discirculant} which is an
analogue of Lemma~\ref{Spectralcirculating} for this case.
This will allow us to simplify the computation of the spectral radius
of $M^-_n$.
As expected, a further consequence of Lemmas~\ref{lem:discirculant}
and \ref{Spectralcirculating} will be that the orientation-reversing
character of $\Phi_{P_n^{^-}}$ in the intervals $I_{y_n}$ and
$I_{y_{2n}}$ has no effects in the entropy. To do this it is
convenient to introduce the notion of \emph{disoriented block
circulant matrix} as follows.
An \emph{$(r,s)-$disoriented block circulant matrix} is a matrix of
the form
\[
A = \begin{pmatrix}
  A_{11} & A_{12} & \dots & A_{1r}\\
  A_{21} & A_{22} & \dots & A_{2r}\\
  \hdotsfor{4}\\
  A_{r1} & A_{r2} & \dots & A_{rr}
\end{pmatrix}
\]
where each $A_{ij}$ is an $s \times s$ matrix for which there exists
an $(r,s)-$block circulant matrix
\[
\widetilde{A} = \begin{pmatrix}
  A_{11} & A_{12} & \dots & A_{1r}\\
  \widetilde{A}_{21} & \widetilde{A}_{22} & \dots & \widetilde{A}_{2r}\\
  \hdotsfor{4}\\
  \widetilde{A}_{r1} & \widetilde{A}_{r2} & \dots & \widetilde{A}_{rr}
\end{pmatrix}
\]
such that given $i\in \{2,\dots, r\}$, either
$A_{ij} = \widetilde{A}_{ij}$ for every $j=1,2,\dots,r$ or
$A_{ij} = \mathbf{J}_s \widetilde{A}_{ij}$ for every $j=1,2,\dots,r.$
That is, every block row of $A$ coincides with the
corresponding block row of $\widetilde{A}$ or is obtained from
the corresponding block row of $\widetilde{A}$ by
pre-multiplying each block by $\mathbf{J}_s.$
Observe that this last operation permutes the individual rows of the
block row symmetrically with respect to the central horizontal axis.
The matrix $\widetilde{A}$ will be called the
\emph{parallelization of $A$}.
Observe that the assumption that the first block row of $A$ and
$\widetilde{A}$ coincide implies that the parallelization of $A$ is
unique.

\begin{lemma}\label{lem:discirculant}
Let
\[
A = \begin{pmatrix}
  A_{11} & A_{12} & \dots & A_{1r}\\
  A_{21} & A_{22} & \dots & A_{2r}\\
  \hdotsfor{4}\\
  A_{r1} & A_{r2} & \dots & A_{rr}
\end{pmatrix}
\]
be a non-negative disoriented $(r,s)-$block circulant matrix
such that
\begin{equation}\label{eq:discircHyp}
  \left( \sum_{j=1}^r A_{1j} \right) \mathbf{J}_s  =
  \mathbf{J}_s \left( \sum_{j=1}^r A_{1j} \right)\ .
\end{equation}
Then
\[
 \rho(A) = \rho\left(\sum_{j=1}^r A_{1j} \right).
\]
\end{lemma}

\begin{proof}
Let
\[
\widetilde{A} = \begin{pmatrix}
  \widetilde{A}_{11} & \widetilde{A}_{12} & \dots & \widetilde{A}_{1r}\\
  \widetilde{A}_{21} & \widetilde{A}_{22} & \dots & \widetilde{A}_{2r}\\
  \hdotsfor{4}\\
  \widetilde{A}_{r1} & \widetilde{A}_{r2} & \dots & \widetilde{A}_{rr}
\end{pmatrix}
\]
be the unique parallelization of $A.$
We will prove that $\norm{A^m}_{\infty} =
\norm{\widetilde{A}^m}_{\infty}$ for
every $m\ge 0$. Then,
\[
\rho(A) =
  \lim_{m\to\infty} \norm{A^m}_{\infty}^{1/m} =
  \lim_{m\to\infty} \norm{\widetilde{A}^m}_{\infty}^{1/m} =
  \rho(\widetilde{A})\ ,
\]
and the result follows from Lemma~\ref{Spectralcirculating}.

For every $m \in \N$ we will write $A^m$ and
$\widetilde{A}^m$ as
\[
\begin{pmatrix}
A^{(m)}_{11} & A^{(m)}_{12} & \dots & A^{(m)}_{1r}\\
A^{(m)}_{21} & A^{(m)}_{22} & \dots & A^{(m)}_{2r}\\
\hdotsfor{4}\\
A^{(m)}_{r1} & A^{(m)}_{r2} & \dots & A^{(m)}_{rr}
\end{pmatrix}
\qquad\text{and}\qquad
\begin{pmatrix}
  \widetilde{A}^{(m)}_{11} & \widetilde{A}^{(m)}_{12} & \dots & \widetilde{A}^{(m)}_{1r}\\
  \widetilde{A}^{(m)}_{21} & \widetilde{A}^{(m)}_{22} & \dots & \widetilde{A}^{(m)}_{2r}\\
  \hdotsfor{4}\\
  \widetilde{A}^{(m)}_{r1} & \widetilde{A}^{(m)}_{r2} & \dots & \widetilde{A}^{(m)}_{rr}
\end{pmatrix}\ ,
\]
respectively. Then,
\begin{align*}
\norm{A^m}_{\infty} &= \max_{i=1,2,\dots,r} \Big\lVert \sum_{j=1}^r A^{(m)}_{ij} \Big\rVert_{\infty}
                     = \max_{i=1,2,\dots,r} \left( \sum_{j=1}^r A^{(m)}_{ij} \right) \mathbf{u}_s
\text{ and, similarly,}\\
\norm{\widetilde{A}^m}_{\infty} &= \max_{i=1,2,\dots,r} \left( \sum_{j=1}^r \widetilde{A}^{(m)}_{ij} \right) \mathbf{u}_s\ ,
\end{align*}
where $\mathbf{u}_s$ denotes the (column) vector of size $s$ with all
the entries equal to 1.
So, to prove the lemma it is enough to show that
\begin{equation}\label{eq:elquecalprovar}
  \left( \sum_{j=1}^r A^{(m)}_{ij} \right) \mathbf{u}_s =
  \left( \sum_{j=1}^r \widetilde{A}^{(m)}_{ij} \right) \mathbf{u}_s
\end{equation}
for every $i=1,2,\dots,r$ and $m \in \N.$

Before proving this claim we will prove two necessary technical
results on the block rows of the matrix $A$ that are
stronger versions of assumption \eqref{eq:discircHyp}.
The first one is the following:
\begin{equation}\label{eq:discircHypArb}
  \left( \sum_{j=1}^r A_{ij} \right) \mathbf{J}_s = \mathbf{J}_s \left( \sum_{j=1}^r A_{ij} \right)
\text{ and }
  \left( \sum_{j=1}^r A_{ij} \right) \mathbf{w}_s = \left( \sum_{j=1}^r A_{rj} \right) \mathbf{w}_s
\end{equation}
for every non-negative vector $\mathbf{w}_s$ of size $s$ such that
$\mathbf{J}_s \mathbf{w}_s = \mathbf{w}_s$
and every $i, r \in \{1,2,\dots,r\}.$

From the definitions of disoriented block circulant matrix,
parallelization and block circulant matrix we get:
\begin{equation}\label{eq:discircdef}
  \sum_{j=1}^r A_{1j} = \sum_{j=1}^r \widetilde{A}_{1j} = \sum_{j=1}^r \widetilde{A}_{ij}
\end{equation}
and $\sum_{j=1}^r A_{ij}$ is either
$\sum_{j=1}^r \widetilde{A}_{ij}$ or
$\mathbf{J}_s \left( \sum_{j=1}^r \widetilde{A}_{ij} \right).$
Thus, from \eqref{eq:discircdef} and \eqref{eq:discircHyp} we
get, respectively,
\[
\sum_{j=1}^r A_{ij} = \begin{cases}
  & \sum_{j=1}^r \widetilde{A}_{ij} = \sum_{j=1}^r A_{1j} \\
  & \mathbf{J}_s \left( \sum_{j=1}^r \widetilde{A}_{ij} \right) =
        \mathbf{J}_s \left( \sum_{j=1}^r A_{1j} \right) =
        \left( \sum_{j=1}^r A_{1j} \right) \mathbf{J}_s.
\end{cases}
\]
Consequently, if $\mathbf{w}_s$ is a non-negative vector of size $s$
such that $\mathbf{J}_s \mathbf{w}_s = \mathbf{w}_s,$ we
have
$
 \left( \sum_{j=1}^r A_{ij} \right) \mathbf{w}_s =
 \left( \sum_{j=1}^r A_{1j} \right) \mathbf{w}_s
$
for every $i \in \{1,2,\dots,r\}.$ This proves the second equality
of \eqref{eq:discircHypArb}. To prove the first one we use the above
expression for $\sum_{j=1}^r A_{ij}$ and \eqref{eq:discircHyp}.
In the first case we have
\[\textstyle
   \left( \sum_{j=1}^r A_{ij} \right) \mathbf{J}_s =
   \left( \sum_{j=1}^r A_{1j} \right) \mathbf{J}_s =
   \mathbf{J}_s \left( \sum_{j=1}^r A_{1j} \right) =
   \mathbf{J}_s \left( \sum_{j=1}^r A_{ij} \right),
\]
and in the second one,
\[\textstyle
   \left( \sum_{j=1}^r A_{ij} \right) \mathbf{J}_s =
   \mathbf{J}_s \left( \sum_{j=1}^r A_{1j} \right) \mathbf{J}_s =
   \mathbf{J}_s \mathbf{J}_s \left( \sum_{j=1}^r A_{1j} \right) =
   \mathbf{J}_s\left( \sum_{j=1}^r A_{ij} \right).
\]
This ends the proof of \eqref{eq:discircHypArb}.

The second technical result that we need is a weaker version of
\eqref{eq:discircHypArb} but that holds for all powers of $A$.
More precisely, for every $i,r \in \{1,2,\dots,r\}$ and $m \in \N,$
\begin{equation}\label{eq:discircWeakHypArbm}
\mathbf{J}_s \left( \sum_{j=1}^r A^{(m)}_{ij} \right) \mathbf{u}_s =
             \left( \sum_{j=1}^r A^{(m)}_{rj} \right) \mathbf{u}_s\ .
\end{equation}
In fact this implies that
$
\left( \sum_{j=1}^r A^{(m)}_{\ell{}j} \right) \mathbf{u}_s
$
is a vector independent on $\ell,$ call it $\mathbf{w}^m_s,$
such that $\mathbf{J}_s \mathbf{w}^m_s = \mathbf{w}^m_s$.

For $m = 1,$ \eqref{eq:discircWeakHypArbm} follows directly from
\eqref{eq:discircHypArb} and the fact that
$\mathbf{J}_s \mathbf{u}_s = \mathbf{u}_s$.
Now we assume that \eqref{eq:discircWeakHypArbm} holds for $m \ge 1$
and prove it for $m+1$.
Clearly,
\begin{equation}\label{eq:powerm}
  \sum_{j=1}^r A^{(m+1)}_{ij} =
  \sum_{j=1}^r \left( \sum_{\ell=1}^r A_{i\ell} A^{(m)}_{\ell{}j} \right) =
  \sum_{\ell=1}^r A_{i\ell} \left( \sum_{j=1}^r A^{(m)}_{\ell{}j} \right)
\end{equation}
for every $i \in \{1,2,\dots,r\}.$
Thus, from the induction hypothesis and \eqref{eq:discircHypArb},
\begin{align*}\textstyle
\mathbf{J}_s \left( \sum_{j=1}^r A^{(m+1)}_{ij} \right) \mathbf{u}_s
   &= \textstyle
      \mathbf{J}_s \sum_{\ell=1}^r A_{i\ell} \left( \sum_{j=1}^r A^{(m)}_{\ell{}j} \right) \mathbf{u}_s =
      \mathbf{J}_s  \big( \sum_{\ell=1}^r A_{i\ell} \big) \mathbf{w}^m_s \\
   &= \textstyle
      \big( \sum_{\ell=1}^r A_{i\ell} \big) \mathbf{J}_s \mathbf{w}^m_s =
      \big( \sum_{\ell=1}^r A_{i\ell} \big) \mathbf{w}^m_s \\
   &= \textstyle
      \big( \sum_{\ell=1}^r A_{r\ell} \big) \mathbf{w}^m_s =
      \sum_{\ell=1}^r A_{r\ell} \left( \sum_{j=1}^r A^{(m)}_{\ell{}j} \right) \mathbf{u}_s \\
   &= \textstyle
      \left( \sum_{j=1}^r A^{(m+1)}_{rj} \right) \mathbf{u}_s.
\end{align*}
This completes the induction step and, thus,
\eqref{eq:discircWeakHypArbm} is proved.

Now we will prove formula \eqref{eq:elquecalprovar} by induction on
$m$ for a fixed but arbitrary  $i \in \{1,2,\dots,r\}.$
Assume that $m = 1.$ If $A_{ij} = \widetilde{A}_{ij}$ for
$j=1,2,\dots,r$ then the equality is trivially true.
Otherwise,
$A_{ij} = \mathbf{J}_s \widetilde{A}_{ij}$ for $j=1,2,\dots,r.$
Hence, from \eqref{eq:discircHypArb},
\begin{align*}\textstyle
\left( \sum_{j=1}^r A_{ij} \right) \mathbf{u}_s
   &= \textstyle
      \left( \sum_{j=1}^r A_{ij} \right) \left( \mathbf{J}_s \mathbf{u}_s \right) =
      \mathbf{J}_s \left( \sum_{j=1}^r A_{ij} \right) \mathbf{u}_s \\
   &= \textstyle
      \mathbf{J}_s \left( \sum_{j=1}^r \mathbf{J}_s \widetilde{A}_{ij} \right) \mathbf{u}_s =
      \left( \sum_{j=1}^r \widetilde{A}_{ij} \right) \mathbf{u}_s,
\end{align*}
because $\mathbf{J}_s$ is an involution
(i.e. $\mathbf{J}^2_s$ is the identity of size $s$).

Assume that \eqref{eq:elquecalprovar} holds for $m \ge 1.$
As above, we will consider two cases.
If $A_{i\ell} = \widetilde{A}_{i\ell}$ for $\ell=1,2,\dots,r,$ then
from \eqref{eq:powerm}, the analogous formula for $\widetilde{A}$ and
the induction assumption we get
\begin{align*}\textstyle
\left( \sum_{j=1}^r A^{(m+1)}_{ij} \right) \mathbf{u}_s
   &= \textstyle
      \sum_{\ell=1}^r A_{i\ell} \left( \sum_{j=1}^r A^{(m)}_{\ell{}j} \right) \mathbf{u}_s \\
   &= \textstyle
      \sum_{\ell=1}^r \widetilde{A}_{i\ell} \left( \sum_{j=1}^r \widetilde{A}^{(m)}_{\ell{}j} \right) \mathbf{u}_s =
      \left( \sum_{j=1}^r \widetilde{A}^{(m+1)}_{ij} \right)
\mathbf{u}_s.
\end{align*}
So, we are left with the case
$A_{i\ell} = \mathbf{J}_s \widetilde{A}_{i\ell}$ for $\ell=1,2,\dots,r.$
In a similar way to the previous case we get
\begin{align*}\textstyle
\left( \sum_{j=1}^r A^{(m+1)}_{ij} \right) \mathbf{u}_s
   &= \textstyle
      \mathbf{J}_s \sum_{\ell=1}^r \widetilde{A}_{i\ell} \left( \sum_{j=1}^r \widetilde{A}^{(m)}_{\ell{}j} \right) \mathbf{u}_s \\
   &= \textstyle
      \mathbf{J}_s \left( \sum_{j=1}^r \widetilde{A}^{(m+1)}_{ij} \right) \mathbf{u}_s.
\end{align*}
Consequently, from \eqref{eq:discircWeakHypArbm} we get:
\[\textstyle
  \left( \sum_{j=1}^r A^{(m+1)}_{ij} \right) \mathbf{u}_s =
  \mathbf{J}_s \left( \sum_{j=1}^r A^{(m+1)}_{ij} \right) \mathbf{u}_s =
  \left( \sum_{j=1}^r \widetilde{A}^{(m+1)}_{ij} \right) \mathbf{u}_s.
\]
This ends the proof of the lemma.
\end{proof}

With the help of Lemma~\ref{lem:discirculant}, as in the
orientable case we obtain:
\begin{corollary}\label{reductionToCompressedisorient}
\[
 \htop(\Phi_{P_n^{^-}})  =
\log\max\left\{\rho\left(\mathsf{C}_n\right),1\right\}.
\]
\end{corollary}

\section{Computation of the topological entropy --- second reduction:
Super compacting the matrix $\mathsf{C}_n$}\label{sec:CtoSC}

The \emph{super compacted matrix of rank $n$} is the $n \times
n$ matrix $\mathsf{SC}_n = (s_{ij})$ defined as follows:
\begin{equation}
s_{ij} = \begin{cases}
1 & \text{if $i \le n-2$ and $j = i+1$ or $i = n$,}\\
2 & \text{if $i = n-2$ and $j = n$,}\\
2n-3 & \text{if $i = n-1$ and $j < n$,}\\
2n-4 & \text{if $i = n-1$ and $j = n$, and}\\
0 & \text{otherwise.}
\end{cases}
\end{equation}
In matrix form we have:
\[
\mathsf{SC}_n = \begin{pmatrix}
0 & 1 & 0 & 0 & \cdots & 0 & 0 & 0\\
0 & 0 & 1 & 0 & \cdots & 0 & 0 & 0\\
\vdots & \vdots & \vdots & \ddots & \cdots & \vdots & \vdots & \vdots\\
\vdots & \vdots & \vdots & \cdots & \ddots & \vdots & \vdots & \vdots\\
0 & 0 & 0 & 0 & \cdots & 1 & 0 & 0\\
0 & 0 & 0 & 0 & \cdots & 0 & 1 & 2\\
2n-3 & 2n-3 & 2n-3 & 2n-3 & \cdots & 2n-3 & 2n-3 & 2n-4\\
1 & 1 & 1 & 1 & \cdots & 1 & 1 & 1
\end{pmatrix}
\]

In this section we prove
\begin{proposition}\label{reductionToSuperCompressed}
For every $n\ge 3,$
\[
 \max\left\{\rho\left(\mathsf{C}_n\right),1\right\} =
 \max\left\{\rho\left(\mathsf{SC}_n\right),1\right\}.
\]
\end{proposition}

To prove the above result we need another intermediate
matrix which we obtain from $\mathsf{C}_n.$
We introduce the \emph{divided compacted matrix of rank $n$}
of size $2n \times 2n,$ denoted by $\mathsf{DC}_n = (d_{ij}),$
which we define as follows:
\begin{equation}
d_{ij} = \begin{cases}
c_{ij} & \text{if $i < n$ and $j \le n$,}\\
c_{i,j-1} & \text{if $i< n$ and $j \ge n+1$,}\\
c_{i-1,j} & \text{if $i > n+1$ and $j \le n$,}\\
c_{i-1,j-1} & \text{if $i > n+1$ and $j \ge n+1$,}\\
1 & \text{if $i = n$ and $j \le n$ or
             $i = n+1$ and $j \ge n+1$,}\\
0 & \text{if $i = n$ and $j \ge n+1$ or
             $i = n$ and $j \le n$,}\\
\end{cases}
\end{equation}
where $\mathsf{C}_n = (c_{ij}).$
In matrix form, $\mathsf{DC}_n$ is
(compare with the definition of the matrix $\mathsf{C}_n$
in Page~\pageref{CM-MF}):
{\fontsize{8}{10}
\[
\setcounter{MaxMatrixCols}{21}\arraycolsep=1pt\def\arraystretch{0.5}
\left(\begin{array}{cccccc|cc|cccccc}
0 & 1 & 0 & 0 & \cdots & 0 & 0 & 0 & 0 & 0 & \cdots & 0 & 0 & 0 \\
0 & 0 & 1 & 0 & \cdots & 0 & 0 & 0 & 0 & 0 & \cdots & 0 & 0 & 0 \\
\vdots & \vdots & \vdots & \ddots & \cdots & \vdots & \vdots &
    \vdots & \vdots & \vdots & \vdots & \vdots & \vdots & \vdots \\
\vdots & \vdots & \vdots & \cdots & \ddots & \vdots & \vdots &
    \vdots & \vdots & \vdots & \vdots & \vdots & \vdots & \vdots \\
0 & 0 & 0 & \cdots & 0 & 1 & 1 & 1 & 0 & 0 & \cdots & 0 & 0 & 0 \\
n-2 & n-2  & n-2 & \cdots & n-2 & n-2 & n-2 & n-2 &
                            n-1 & n-1 & \cdots & n-1 & n-1 & n-1 \\
\hline\rule{0pt}{8pt}
1 & 1 & 1 & \cdots & 1 & 1 & 1 & 0 & 0 & \cdots & 0 & 0 & 0 & 0\\
0 & 0 & 0 & \cdots & 0 & 0 & 0 & 1 & 1 & \cdots & 1 & 1 & 1 & 1\\
\hline\rule{0pt}{8pt}
n-1 & n-1 & n-1 & \cdots & n-1 & n-1 &
      n-2 & n-2 & n-2 & n-2 & \cdots & n-2 & n-2 & n-2 \\
0 & 0 & 0 & \cdots & 0 & 0 & 1 & 1 & 1 & 0 & \cdots & 0 & 0 & 0 \\
\vdots & \vdots & \vdots & \vdots & \vdots & \vdots & \vdots &
    \vdots &\vdots & \ddots & \cdots  & \vdots & \vdots & \vdots\\
\vdots & \vdots & \vdots & \vdots & \vdots & \vdots & \vdots &
    \vdots &\vdots & \ddots & \cdots  & \vdots & \vdots & \vdots\\
0 & 0 & 0 & 0 & \cdots & 0 & 0 & 0 & 0 & 0 & \cdots & 1 & 0 & 0 \\
0 & 0 & 0 & 0 & \cdots & 0 & 0 & 0 & 0 & 0 & \cdots & 0 & 1 & 0 \\
\end{array}\right)
\]}

Notice that the matrix $\mathsf{DC}_n$ is indeed the
Markov matrix of a topological model obtained by subdividing the
central interval of the topological model from Remark~\ref{TP:C} at a
fixed point (that exists because the central interval covers itself).

\begin{proof}[Proof of Proposition~\ref{reductionToSuperCompressed}]
First we will prove that
\[
\mathrm{Spec}(\mathsf{DC}_n) = \mathrm{Spec}(\mathsf{C}_n) \cup \{1\},
\]
where $\mathrm{Spec}(\cdot)$ denotes the set of eigenvalues of a
matrix.
To do this observe that $1$ is an eigenvalue of $\mathsf{DC}_n$ with
eigenvector $(0,0,\dots,0,1,-1,0\dots,0,0),$ where the two non-zero
elements of this vector are at the $n$ and $n+1$ entries.
Also, if $\mu$ is an eigenvalue of $\mathsf{C}_n$ with eigenvector
$(v_1, v_2,\dots, v_{2n-1}),$ then, it follows directly from the
definition of the matrix $\mathsf{DC}_n$ that $\mu$ is also an
eigenvalue of $\mathsf{DC}_n$ with eigenvector
\[
\left(v_1, v_2,\dots, v_{n-1},
   \tfrac{v_{n}}{2},\tfrac{v_{n}}{2},
   v_{n+1},\dots,v_{2n-1}\right).
\]
Conversely, if $\mu$ is an eigenvalue of $\mathsf{DC}_n$ with
eigenvector $(v_1, v_2,\dots, v_{2n}),$ then, again from the
definition of the matrix $\mathsf{DC}_n,$ it follows that $\mu$
is also an eigenvalue of $\mathsf{C}_n$ with eigenvector
\[
\left(v_1, v_2,\dots, v_{n-1},
   v_{n} + v_{n+1},
  v_{n+2},\dots,v_{2n}\right).
\]
This proves the statement.

The second step of the proof will be to show that
$\rho\left(\mathsf{DC}_n\right) = \rho\left(\mathsf{SC}_n\right).$
To do this it is convenient to write the matrix $\mathsf{DC}_n$ in
block form, with blocks of size $n \times n$:
\[
\mathsf{DC}_n =
\begin{pmatrix}
D_{11} & D_{12}\\
D_{21} & D_{22}\\
\end{pmatrix} .
\]
Observe that $\mathsf{DC}_n$ is symmetric with respect to the central
point. So,
$D_{21} = \mathbf{J}_{n} D_{12} \mathbf{J}_{n}$ and
$D_{22} = \mathbf{J}_{n} D_{11} \mathbf{J}_{n},$
which amounts to:
\[
\mathsf{DC}_n =
\begin{pmatrix}
D_{11} & D_{12}\\
\mathbf{J}_{n} D_{12} \mathbf{J}_{n} &
\mathbf{J}_{n} D_{11} \mathbf{J}_{n}\\
\end{pmatrix} .
\]

Now let us consider the block matrix of size $2n \times 2n$ defined
by
\[
 \mathbf{Z}:= \begin{pmatrix}
               \mathbf{I}_{n} & \mathbf{0}_{n}\\
               \mathbf{0}_{n} & \mathbf{J}_{n}
              \end{pmatrix},
\]
where $\mathbf{I}_{n}$ denotes the identity matrix of size
$n \times n.$
Clearly $\mathbf{Z}$ is non-singular and
$\mathbf{Z}^{-1} = \mathbf{Z}.$
Hence,
\[
 \mathbf{Z} \mathsf{DC}_n \mathbf{Z}^{-1} = \begin{pmatrix}
D_{11} & D_{12} \mathbf{J}_{n} \\
D_{12} \mathbf{J}_{n} & D_{11}
\end{pmatrix}
\]
because
$\mathbf{J}_{n} \mathbf{J}_{n} = \mathbf{I}_{n}.$
Observe that
$\mathbf{Z} \mathsf{DC}_n \mathbf{Z}^{-1}$
is a non-negative block circulant matrix. Thus,
\[
 \rho\left(\mathsf{DC}_n\right) =
 \rho\left(\mathbf{Z} \mathsf{DC}_n \mathbf{Z}^{-1}\right) =
 \rho\left(D_{11} + D_{12} \mathbf{J}_{n}\right)
\]
by Lemma~\ref{Spectralcirculating}.
Moreover, by direct inspection it follows that
$D_{11} + D_{12} \mathbf{J}_{n} = \mathsf{SC}_n.$

Summarizing, we have proved
\[
 \max\left\{\rho\left(\mathsf{C}_n\right),1\right\} =
 \max\left\{\rho\left(\mathsf{DC}_n\right),1\right\} =
 \max\left\{\rho\left(\mathsf{SC}_n\right),1\right\}.
\]
\end{proof}

\begin{remark}\label{TP:SC}
The topological model in Remark~\ref{TP:C} has a fixed point and
commutes with the symmetry of degree -1 with respect to this fixed
point.
The quotient space obtained by identifying each orbit of the symmetry to a
point is a closed interval, and the induced map on this quotient space is also
a Markov map. The matrix $\mathsf{SC}_n$ is in fact the Markov matrix
of this quotient map.
\end{remark}

\section{The spectral radius of $\mathsf{SC}_n$ and proof of
Theorem~\ref{MainTh}}\label{sec:Rome}

To compute the spectral radius of $\mathsf{SC}_n$ we will use the
\emph{rome method} proposed in \cite{BGMY}.
To this end we have to introduce some notation.

Let $M = (m_{ij})$ be a $k\times k$ matrix.
Given a sequence $p = (p_j)_{j=0}^{\ell}$ of elements of
$\{1,2,\dots,k\}$ we define the \emph{width of $p$,} denoted by
$w(p),$  as the number $\prod_{j=1}^{\ell} m_{p_{j-1}p_j}$.
If $w(p)\ne 0$ then $p$ is called a \emph{path of length $\ell$.}
The length of a path $p$ will be denoted by $\ell(p)$.
A loop is a path such that $p_{\ell} = p_0$ i.e.\ that begins and ends at
the same index.

A subset $R$ of $\{1,2,\dots,k\}$ is called a \emph{rome\/}
if there is no loop outside $R$, i.e.\ there is no path
$(p_j)_{j=0}^{\ell}$ such that $p_{\ell} = p_0$ and
$\set{p_j}{0\le j\le \ell}$ is disjoint from $R$.
For a rome $R$ a path $(p_j)_{j=0}^{\ell}$ is called \emph{simple\/} if
$p_i\in R$ for $i=0,\ell$ and $p_i\notin R$ for $i=1,2,\dots,\ell-1$.

If $R=\{r_1,r_2,\dots,r_{\ell}\}$ is a rome of a matrix $M$ then we
define an $\ell\times\ell$ matrix-valued real function $M_R(x)$ by
setting $M_R(x) = (a_{ij}(x))$, where $a_{ij}(x) = \sum_p w(p)\cdot
x^{-\ell(p)}$, where the summation is over all simple paths
originating at $r_i$ and terminating at $r_j$.

\begin{theorem}[Theorem 1.7 of \cite{BGMY}]\label{rome}
If $R$ is a rome of cardinality $\ell$ of a $k\times k$ matrix $M$ then
the characteristic polynomial of $M$ is equal to
\[
 (-1)^{k-\ell} x^k \det(M_R(x) - \mathbf{I}_{\ell}).
\]
\end{theorem}

To use Theorem~\ref{rome} it is helpful to represent the
matrix $M$ in form of a combinatorial graph which amounts to draw all
paths of length 1 associated to $M.$ To do this we introduce the
following notation. A path $(i,j)$ of length 1 will be written
as $i \labelarrow{w} j,$ where $w$ denotes the width of the
path.
For the matrix $M$ the width $w$ of the path
$i \labelarrow{w} j,$ is just the entry $(M)_{i,j} \ne 0$.
Observe that, with this notation, a path $p=(p_j)_{j=0}^k$ is
written as
\[
 p_0 \labelarrow{w_0} p_1 \labelarrow{w_1} \cdots p_{k-1}
\labelarrow{w_{k-1}} p_k
\]
and $w(p) = \prod_{i=0}^{k-1} w_i$.

We will compute the spectral radius of $\mathsf{SC}_n$ in
Proposition~\ref{polynomial} by using Theorem~\ref{rome}.
In Figure~\ref{GraphSCg} we show the combinatorial graph associated to
$\mathsf{SC}_n.$
\newcommand{\wwcolor}[1]{\textcolor{black}{#1}}
\begin{figure}
\tikzstyle{place}=[ellipse,draw=black!50,fill=black!20,thick]
\tikzstyle{post}=[->,shorten >=1pt,>=stealth,semithick]
\begin{tikzpicture}
\foreach \p in {1,2,3}{ \node[place]  (\p )  at (2*\p ,0) { \p}; }
\node (4) at (8,0) {$\cdots$};
\node[place] (2gm2) at (10,0) {$n-2$};

\node (phantom) at (12.5,0) {};

\node[place] (2g) [below of=phantom] {$n$};
\node[place] (2gm1) [above of=phantom] {$n-1$};

\draw[decorate, very thick, decoration={brace,amplitude=5pt, mirror, raise=15pt}]
               (1.west) -- (2gm2.east) node[midway,yshift=-21pt] (LB) {};
\draw[decorate, very thick, decoration={brace,amplitude=5pt, raise=15pt}]
               (1.west) -- (2gm2.east) node[midway,yshift=21pt] (UB) {};

\path[post] (1.east)+(1pt,0) edge node[auto] {\wwcolor{1}} (2)
            (2.east)+(1pt,0) edge node[auto] {\wwcolor{1}} (3)
            (3.east)+(1pt,0) edge node[auto] {\wwcolor{1}} (4)
            (4.east)+(1pt,0) edge node[auto] {\wwcolor{1}} (2gm2)
            (2gm2.east)+(1pt,0) edge node[auto] {\wwcolor{1}} (2gm1)
            (2gm2.east)+(1pt,0) edge node[auto,swap] {\wwcolor{2}} (2g)
            (2g.north)+(0,1pt) edge node[auto] {\wwcolor{1}} (2gm1)
            (2g) edge [in=300,out=240,loop] node[auto,swap] {\wwcolor{1}} ()
            (2gm1.south)+(5pt,-2pt) edge[bend left] node[auto] {\wwcolor{$2n-4$}} (2g)
            (2gm1) edge [in=60,out=120,loop] node[auto] {\wwcolor{$2n-3$}} ()
            (2g.west)+(-1pt,0) edge[bend left] node[auto] {\wwcolor{1}} (LB.center)
            (2gm1.west)+(-1pt,0) edge[bend right] node[auto,swap] {\wwcolor{$2n-3$}} (UB.center);
\end{tikzpicture}
\caption{The combinatorial graph associated to $\mathsf{SC}_n.$
The arrows ending at braces indicate multiple arrows with the
same weight, each one directed towards a node under the
brace.}\label{GraphSCg}
\end{figure}
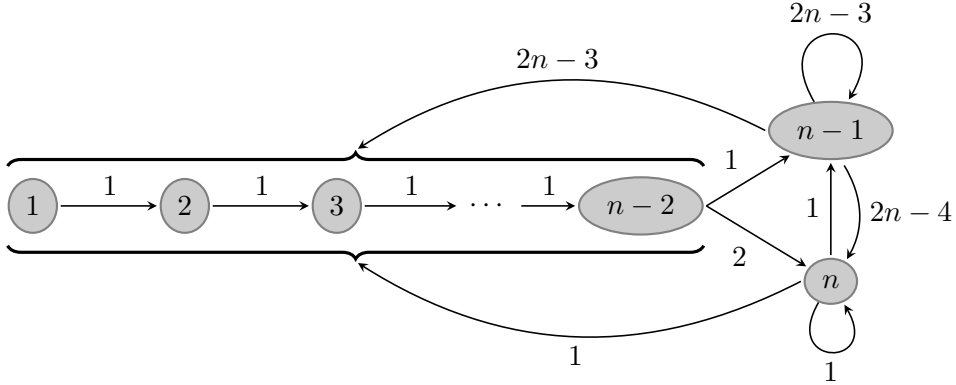

\begin{remark}
The combinatorial graph associated to $\mathsf{SC}_n$ shown in
Figure~\ref{GraphSCg} is in fact the generalized Markov graph of the
topological model obtained in Remark~\ref{TP:SC}.
\end{remark}

\begin{proposition}\label{polynomial}
The spectral radius of $\mathsf{SC}_n$ is the largest root of the
polynomial $Q_n(x).$
\end{proposition}

\begin{proof}
By direct inspection of the graph of Figure~\ref{GraphSCg} it
follows that $\mathsf{SC}_n$ is an irreducible non-negative integer
matrix. Hence, by the Perron-Frobenius Theorem (see \cite{Gant}), we
get that the spectral radius of $\mathsf{SC}_n$ is the largest
eigenvalue of the characteristic polynomial of $\mathsf{SC}_n.$ It
is larger than 1 and simple. So, to prove the theorem, we have to
show that the characteristic polynomial of $\mathsf{SC}_n$ is
$Q_{n}(x).$

Clearly $R = \{n-1,n\}$ is a rome of $\mathsf{SC}_n$ (see
Figure~\ref{GraphSCg}). Hence,
\[
M_R(x) = \begin{pmatrix}
\beta\bigl(x^{-1} + z(x) \bigr) & (\beta-1)x^{-1} + 2 \beta z(x) \\
 x^{-1} + z(x) & x^{-1} + 2z(x)
\end{pmatrix}\,
\]
where $\beta = 2n-3,$ $z(x):= \sum_{\ell=2}^{n-1} x^{-\ell}.$

By Theorem~\ref{rome}, the characteristic polynomial of
$\mathsf{SC}_n$ is
\begin{align*}
(-1)^{n-2} & x^{n} \begin{vmatrix}
\beta\bigl(x^{-1} + z(x) \bigr) - 1 &
     2\beta\bigl(x^{-1} + z(x) \bigr) - (\beta + 1)x^{-1}\\
 x^{-1} + z(x) & 2\bigl(x^{-1} + z(x) \bigr) - x^{-1} - 1
\end{vmatrix}\\
& = x^{n}
    \Bigl(
       \bigl(x^{-1} - \beta - 2\bigr)
       \bigl(x^{-1} + z(x) \bigr) + x^{-1} + 1
    \Bigr)\\
& = x^{n}
    \bigl(x^{-1} - (2n-1)\bigr) \left(\sum_{\ell=1}^{n-1} x^{-\ell}\right)
    + x^{n-1} + x^{n}\\
& = x^{n} - (2n - 2) \sum_{j=1}^{n-1} x^{j} + 1.
\end{align*}
This ends the proof of the proposition.
\end{proof}

Next we prove a technical lemma that studies the polynomial
\eqref{thepolynomial} and gives the bounds for $\lambda_n$.

\begin{lemma}\label{entropybounds}
For every $n \ge 3,$ $Q_n(x)$  has a unique real root
$\lambda_n$ larger than one. Moreover, for $n \ge 4,$
\[
 2n-1 - \frac{1}{(2n-1)^{n-2}} < \lambda_n < 2n-1.
\]
\end{lemma}

\begin{proof}
Observe that
$Q_n(0) = 1,$
$Q_n(1) = -2n(n-2) < 0$ and
\[
 Q'_n(x) = n\left(
    x^{n-1} - \frac{2(n - 1)}{n} \sum_{j=1}^{n-1} jx^{j-1}\right)
\le n\left( x^{n-1} - \frac{2(n - 1)}{n} \right).
\]
Since $x^{n-1} - \tfrac{2(n - 1)}{n}$ is negative for every for $n \ge 3$
and $x \in [0,1],$ it follows that $Q_n$ has a unique root in $(0,1)$.

On the other hand, it is easy to see that
$Q_n(x) = x^n Q_n(x^{-1})$
(that is, $Q_n$ is a \emph{reciprocal polynomial}) and, hence,
$z$ is a root of $Q_n$ if and only if $z^{-1}$ is a root of $Q_n$.
Consequently, $Q_n(x)$ has a unique real root larger than one.

Also,
\[
Q_n(2n-1) =
(2n-1)^n - 2(n-1) \frac{(2n-1)^n-(2n-1)}{2(n-1)} + 1 =
2n.
\]
So, $\lambda_n < 2n-1.$

To end the proof of the lemma it is enough to show that
\[
Q_n\left(2n-1 - \tfrac{1}{(2n-1)^{n-2}}\right) < 0
\]
for $n \ge 4.$
We have
\begin{align*}
Q_n(z) &= z^n - 2(n-1)s \frac{z^n-z}{2(n-1)s -1} + 1 \\
       &= -\frac{z^n - 4n(n-1)s + 2n - 1}{2(n-1)s -1},
\end{align*}
where
$z := 2n-1 - \tfrac{1}{s}$ and
$s := (2n-1)^{n-2}.$

Since $2(n-1)s - 1 > 2n - 1 > 0,$
$z^n - 4n(n-1)s > 0$ implies that $Q_n(z) < 0.$

An exercise shows that $z^n > (2n-1)^{n} - n(2n-1).$
Hence,
\[
 z^n - 4n(n-1)s >
  (2n-1)^{n-2} \left(1 - \frac{n}{(2n-1)^{n-3}} \right)
\]
and $z^n - 4n(n-1)s > 0$ for $n \ge 4.$
So, $Q_n(z) < 0$ and the lemma is proved.
\end{proof}

\begin{proof}[Proof of Theorem~\ref{MainTh}]
It follows from
Theorem~\ref{old-theo:2.3},
Corollaries~\ref{reductionToCompressed} and
\ref{reductionToCompressedisorient},
Propositions~\ref{reductionToSuperCompressed} and \ref{polynomial},
and Lemma~\ref{entropybounds}.

\end{proof}

\bibliographystyle{plain}
\bibliography{HvolAsymp}
\end{document}